\documentclass[a4paper,11pt,reqno,english]{smfart}

\usepackage{enumerate}
\usepackage{amssymb,amsmath,latexsym,amsthm}
\usepackage{CJKutf8}
\usepackage[a4paper, tmargin=0.9in, bmargin=0.8in, lmargin=0.8in,rmargin=0.8in]{geometry}
\usepackage{url}
\usepackage[main=english]{babel}
\usepackage[utf8]{inputenc}
\usepackage{mathrsfs}
\usepackage{xcolor}
\usepackage{bigints}
\usepackage{comment}
\definecolor{violet}{rgb}{0.0,0.2,0.7}
\definecolor{rouge2}{rgb}{0.8,0.0,0.2}
\usepackage{tikz}
\usepackage{empheq}
\usepackage{tikz-cd}
\usetikzlibrary{matrix,arrows,decorations.pathmorphing}
\usepackage{hyperref}
\usepackage{mathpazo} 
\hypersetup{
	bookmarks=true,         
	unicode=false,          
	pdftoolbar=true,        
	pdfmenubar=true,        
	pdffitwindow=false,     
	pdfstartview={FitH},    
	pdftitle={},    
	pdfauthor={},     
	colorlinks=true,       
	linkcolor=rouge2,          
	citecolor=violet,        
	filecolor=black,      
	urlcolor=cyan}           
\usepackage{enumitem}
\usepackage{appendix}

\usepackage{mathtools}
\usepackage{dsfont}

\theoremstyle{plain}
    \newtheorem{thm}{Theorem}[section]
    
    \newtheorem{lem}[thm]{Lemma}
    \newtheorem{prop}[thm]{Proposition}
    
    \newtheorem{setupnota}[thm]{Setup and Notations}
    \newtheorem{cor}[thm]{Corollary}
    
    \newtheorem*{claim*}{Claim}
    \newtheorem{claim}[thm]{Claim}

\theoremstyle{plain}
    \newtheorem{bigthm}{Theorem}
    \newtheorem*{mainthm*}{Main Theorem}

    \newtheorem{taggedbigthmx}{Theorem} 
    \newenvironment{taggedbigthm}[1]
    {\renewcommand\thetaggedbigthmx{#1}\taggedbigthmx}
    {\endtaggedbigthmx}
    
    \newtheorem*{bigrmk*}{Remark}

\theoremstyle{definition}
    \newtheorem{defn}[thm]{Definition}
    \newtheorem*{defn*}{Definition}

    \newtheorem*{conj*}{Conjecture}
    \newtheorem*{ack*}{Acknowledgements}

\theoremstyle{remark}
    
    \newtheorem*{rmk*}{Remark}

    \newtheorem*{ques*}{Question}
    \newtheorem*{ans*}{Answer}

\numberwithin{equation}{section}
\renewcommand{\labelenumi}{(\roman{enumi})}
\newlist{steps}{enumerate}{1}
\setlist[steps, 1]{label = Step \arabic*:}

\mathcode`l="8000
\begingroup
\makeatletter
\lccode`\~=`\l
\DeclareMathSymbol{\lsb@l}{\mathalpha}{letters}{`l}
\lowercase{\gdef~{\ifnum\the\mathgroup=\m@ne \ell \else \lsb@l \fi}}%
\endgroup

\DeclareFontFamily{U}{MnSymbolC}{}
\DeclareSymbolFont{MnSyC}{U}{MnSymbolC}{m}{n}
\DeclareFontShape{U}{MnSymbolC}{m}{n}{
	<-6>  MnSymbolC5
	<6-7>  MnSymbolC6
	<7-8>  MnSymbolC7
	<8-9>  MnSymbolC8
	<9-10> MnSymbolC9
	<10-12> MnSymbolC10
	<12->   MnSymbolC12}{}
\DeclareMathSymbol{\intprod}{\mathbin}{MnSyC}{'270}
\DeclareMathOperator{\Id}{Id}
\DeclareMathOperator{\Hom}{Hom}

\DeclareMathOperator{\im}{im}
\DeclareMathOperator{\tr}{tr}
\DeclareMathOperator{\Tr}{Tr}
\DeclareMathOperator{\pr}{pr}
\DeclareMathOperator{\Vol}{Vol}

\DeclareMathOperator{\supp}{supp}

\DeclareMathOperator{\dist}{dist}

\DeclareMathOperator{\SH}{SH}
\DeclareMathOperator{\Har}{H}
\DeclareMathOperator{\PSH}{PSH}

\DeclareMathOperator{\Exc}{Exc}

\DeclareMathOperator{\rk}{rank}

\DeclareMathOperator{\HE}{HE}

\DeclareMathOperator*{\esssup}{ess\,sup}

\DeclareMathOperator*{\esslimsup}{ess\,lim\,sup}
\DeclareMathOperator{\Tor}{Tor}
\DeclareMathOperator{\End}{End}
\DeclareMathOperator{\Sing}{Sing}

\DeclareMathOperator{\BCC}{c_1^{BC}}
\DeclareMathOperator{\Herm}{Herm}
\DeclareMathOperator{\Gr}{Gr}

\def\1{\mathds{1}}

\def\L{\mathbf{L}}

\newcommand{\ii}{\mathrm{i}}

\newcommand{\loc}{\mathrm{loc}}

\newcommand{\wOm}{{\widetilde{\Omega}}}
\newcommand{\wom}{{\widetilde{\omega}}}

\newcommand{\homg}{{\widehat{\omega}}}

\newcommand{\hV}{{\widehat{V}}}
\newcommand{\hX}{{\widehat{X}}}

\newcommand\sm{\sigma}
\newcommand\dt{\delta}

\newcommand\vep{\varepsilon}
\newcommand\vph{\varphi}

\newcommand\om{\omega}
\newcommand\ta{\theta}
\newcommand\gm{\gamma}

\newcommand\af{\alpha}
\newcommand\bt{\beta}
\newcommand\ld{\lambda}

\newcommand\Dt{\Delta}
\newcommand\Om{\Omega}

\newcommand\Gm{\Gamma}

\newcommand\Ta{\Theta}

\newcommand\RA{\mathrm{A}}
\newcommand\RBC{\mathrm{BC}}

\newcommand\RG{\mathrm{G}}

\newcommand\BN{\mathbb{N}}

\newcommand\BR{\mathbb{R}}
\newcommand\BC{\mathbb{C}}
\newcommand\BB{\mathbb{B}}

\newcommand\BD{\mathbb{D}}

\newcommand\CA{\mathcal{A}}

\newcommand\CC{\mathcal{C}}

\newcommand\CE{\mathcal{E}}
\newcommand\CF{\mathcal{F}}
\newcommand\CG{\mathcal{G}}
\newcommand\CH{\mathcal{H}}

\newcommand\CK{\mathcal{K}}

\newcommand\CO{\mathcal{O}}

\newcommand\CS{\mathcal{S}}

\newcommand\CU{\mathcal{U}}

\newcommand\lt{\left}
\newcommand\rt{\right}

\newcommand\ra{\rightarrow}

\newcommand\pl{\partial}
\newcommand\db{\bar{\partial}}

\newcommand\ddb{\partial \bar{\partial}}

\newcommand\dd{\mathrm{d}}
\newcommand\dc{\mathrm{d}^{\mathrm{c}}}
\newcommand\ddc{\mathrm{d}\mathrm{d}^{\mathrm{c}}}

\newcommand\norm[1]{\left\lVert {#1} \right\rVert}

\newcommand\abs[1]{\left\lvert {#1} \right\rvert}

\newcommand\w{\wedge}

\newcommand\reg{\mathrm{reg}}
\newcommand\sing{\mathrm{sing}}

\newcommand\set[2]{\left\{ {#1} \, \middle| \, {#2} \right\}}
\newcommand\iprod[2]{\left\langle {#1}, {#2} \right\rangle}

\newcommand{\RN}[1]{\textup{\uppercase\expandafter{\romannumeral#1}}}

\makeatletter
\newsavebox{\@brx}
\newcommand{\llangle}[1][]{\savebox{\@brx}{\(\m@th{#1\langle}\)}%
  \mathopen{\copy\@brx\kern-0.5\wd\@brx\usebox{\@brx}}}
\newcommand{\rrangle}[1][]{\savebox{\@brx}{\(\m@th{#1\rangle}\)}%
  \mathclose{\copy\@brx\kern-0.5\wd\@brx\usebox{\@brx}}}
\makeatother

\title{Gauduchon metrics and Hermite--Einstein metrics on non-K\"ahler varieties}

\author{Chung-Ming Pan}
\address[Chung-Ming Pan]{Centre interuniversitaire de recherches en g\'eom\'etrie et topologies (CIRGET); Universit\'e du Qu\'ebec \`a Montr\'eal; Case postale 8888, Succursale centre-ville, Montr\'eal, Qu\'ebec, H3C 3P8, Canada}
\email{\href{mailto:pan.chung_ming@uqam.ca}{pan.chung\_ming@uqam.ca} \qquad\qquad\qquad\qquad\qquad\qquad\qquad\qquad\qquad\qquad\qquad\qquad\qquad\qquad}
\urladdr{\href{https://chungmingpan.github.io/}{https://chungmingpan.github.io/}}

\date{\today}
\subjclass{Primary: 53C55, 53C07, 32C15, Secondary: 14F06, 32W20, 35J08}
\keywords{Singular Gauduchon metrics, Singular Hermite--Einstein metrics, Stable reflexive sheaves, Uniform Sobolev inequality}

\begin{document} 

\maketitle
\begin{abstract} 
We show the existence of Gauduchon metrics on arbitrary compact hermitian varieties, generalizing our previous work on smoothable singularities. 
These metrics allow us to define the notion of slope stability for torsion-free coherent sheaves on compact normal varieties that are not necessarily K\"ahler. 
Then we prove the existence and uniqueness of singular Hermite--Einstein metrics for slope-stable reflexive sheaves on non-K\"ahler normal varieties. 
\end{abstract}

\tableofcontents

\section*{Introduction}
Constructing special metrics on complex manifolds is a central problem in complex geometry.
Yau's celebrated solution of the Calabi conjecture \cite{Yau_1978} and Donaldson--Uhlenbeck--Yau theorem \cite{Donaldson_1985, Uhlenbeck_Yau_1986, Donaldson_1987} on the existence of Hermite--Einstein metrics over stable vector bundles are landmarks in K\"ahler geometry. 

\smallskip
A great deal of attention has also been devoted to studying non-K\"aher geometry. 
In this context, an important question is how much of the theories on special metrics can be extended from K\"ahler to non-K\"ahler geometry. 
Gauduchon metrics are crucial in these investigations. 
A metric $\om_\RG$ on an $n$-dimensional ($n \geq 2$) complex manifold $X$ is called Gauduchon if $\ddc \om_\RG^{n-1} = 0$. 
A classical result by Gauduchon \cite{Gauduchon_1977} states that, on a compact complex manifold, every hermitian metric is conformally equivalent to a Gauduchon metric, and such a metric is moreover unique up to constant scaling. 
Li and Yau \cite{Li_Yau_1987} used Gauduchon metrics to define slope stability on non-K\"ahler manifolds and generalized the Donaldson--Uhlenbeck--Yau theorem to this setting. 
This plays an important role in classifying surfaces of class VII (see e.g. \cite{Teleman_2019} and references therein). 
On the other hand, motivated by Yau's theorem, Gauduchon \cite[IV.5]{Gauduchon_1984} posed a Calabi--Yau type question concerning Gauduchon metrics, later solved by Sz\'ekelyhidi, Tosatti, and Weinkove \cite{STW_2017}. 

\smallskip
Singular spaces frequently arise from classification theories of complex manifolds, such as Minimal Model Program and moduli theory. 
Therefore, studying canonical metrics on singular varieties has become more and more important in recent decades. 
Generalizing Yau's theorem, singular K\"ahler--Einstein metrics with non-positive curvature on mildly singular K\"ahler varieties have been constructed by Eyssidieux--Guedj--Zeriahi \cite{EGZ_2009}. 
On the other hand, Bando and Siu \cite{Bando_Siu_1994} extended Donaldson--Uhlenbeck--Yau theorem to stable reflexive sheaves over compact K\"ahler manifolds. 
This has recently been further generalized to normal K\"ahler spaces \cite{Chen_Wentworth_2024, Chen_2022, Paun_et_al_2023_HE_sing}. 

\smallskip
In non-K\"ahler geometry, exploring singular objects is also significant. 
Reid's fantasy \cite{Reid_1987} conjectures that Calabi--Yau threefolds should form a connected moduli space via certain geometric transitions introduced by Clemens and Friedman \cite{Clemens_1983, Friedman_1986}, which allows deformations and bimeromorphic transformations through non-K\"ahler Calabi--Yau varieties. 
This subject has garnered much interest (see \cite{Friedman_1991, Lu_Tian_1994, Chuan_2012, Fu_Li_Yau_2012, Collins_Picard_Yau_2024} and the references therein). 

\smallskip
It is therefore essential to understand the construction of Gauduchon metrics on singular non-K\"ahler varieties and to investigate their applications. 
In this article, we contribute to the following two directions: 
\begin{itemize}
    \item {\bf Singular Gauduchon metrics:} 
    We prove that on any given compact hermitian variety, there exists a conformally equivalent singular Gauduchon metric, which is moreover unique up to constant scaling. 
    \item {\bf Singular Hermite--Einstein metrics:} 
    Using singular Gauduchon metrics, one can make sense of slope stability for torsion-free coherent sheaves on arbitrary compact normal varieties, not necessarily K\"ahler. 
    Given a stable reflexive sheaf on such a variety, we establish the existence and uniqueness of a singular Hermite--Einstein metrics. 
\end{itemize}

\subsection*{Singular Gauduchon metrics}
In \cite[Thm.~A]{Pan_2022}, the author has shown the existence of Gauduchon metrics on hermitian varieties that are {\it smoothable}.
The first aim of this article is to get rid of this extra assumption and prove the singular version of Gauduchon's theorem in full generality. 

\begin{bigthm}
Let $(X,\om)$ be an $n$-dimensional compact hermitian variety.
There exists a Gauduchon metric $\om_\RG$ that is conformal and quasi-isometric to $\om$.
Moreover, it is unique up to constant scaling in the conformal and quasi-isometric class of $\om$. 
\end{bigthm}

Specifically, the metric $\om_\RG$ is constructed via an approximation process as follows:
due to Hironaka's theorem \cite{Hironaka_1964}, there is $\pi: \hX \to X$ a log resolution of singularities consisting of finitely many blowups of smooth centers, and the map $\pi$ is an isomorphism away from the singular points of $X$. 
Denote by $D$ the exceptional divisor of $\pi$.
One can find a closed $(1,1)$-form $\ta$ such that $\homg := \pi^\ast \om + \ta$ is a hermitian metric on $\hX$. 
Consider a family of hermitian metrics $\om_\vep := \pi^\ast \om + \vep \homg$ on $\hX$ for $\vep \in (0,1]$. 
By Gauduchon's theorem \cite{Gauduchon_1977}, there exists a unique smooth positive function $\rho_\vep$ that satisfies $\ddc (\rho_\vep \om_\vep^{n-1}) = 0$ and $\inf_\hX \rho_\vep = 1$. 
Then we establish the following uniform $L^\infty$-estimate and convergence: 

\begin{taggedbigthm}{A'}[cf. Theorem~\ref{thm:existence_bdd_Gauduchon}]\label{bigthm:approx_and_conv}
Under the above setup, the following properties hold:
\begin{enumerate}
    \item There is a uniform constant $C_G > 0$ independent of $\vep \in (0,1]$ such that $\sup_\hX \rho_\vep \leq C_G$;
    \item As $\vep \to 0$, one can extract a subsequence $\rho_{\vep_k}$ converging in $\CC^\infty_{\loc}(\hX \setminus D)$ towards a function $\rho \in \CC^\infty_\loc(\hX \setminus D) \cap L^\infty(\hX \setminus D)$ which is also bounded away from zero.
\end{enumerate}
\end{taggedbigthm}

In particular, since $\pi: \hX \setminus D \to X^\reg$ is an isomorphism, $\rho$ can be identified as a smooth function $X^\reg$, which we still denote by $\rho$. 
Then $\om_\RG := \rho^{\frac{1}{n-1}} \om$ is a bounded Gauduchon metric on $X^\reg$. 
A bounded Gauduchon metric refers to a Gauduchon metric $\om_\RG$ on $X^\reg$, which is conformally equivalent and quasi-isometric to $\om$ on $X^\reg$. 
Furthermore, $\rho$ is the unique bounded solution to $\ddc(\rho \om^{n-1}) = 0$ on $X^\reg$ up to scaling. 

\smallskip
A difficulty in proving the above result lies in establishing a uniform Sobolev inequality for $(\om_\vep)_{\vep \in (0,1]}$. 
Guo--Phong--Song--Sturm have established uniform Sobolev inequalities for K\"ahler metrics following a series of development \cite{Guo_Phong_Sturm_2024, Guo_Phong_Song_Sturm_2024_diameter, Guo_Phong_Song_Sturm_2023_Sobolev}. 
However, in the hermitian setting, their techniques cannot be directly applied, and additionally, the formulation of Green’s functions for the complex Laplacian is much more complicated than in the K\"ahler case. 

\smallskip
We use a localizing and patching method to establish a uniform Sobolev inequality for $(\om_\vep)_{\vep \in (0,1]}$. 

\begin{bigthm}[cf. Theorem~\ref{thm:unif_Sobolev_Poincare}]\label{bigthm:unif_Sobolev}
Let $(X,\om)$ be an $n$-dimensional compact hermitian variety and let $\pi: \hX \to X$ be a log resolution of singularities. 
Fix a hermitian metric $\homg$ on $\hX$ and consider an $\vep$-family of hermitian metrics $\om_\vep := \pi^\ast \om + \vep \homg$ for $\vep \in (0,1]$. 
Then for any $q \in (1, \frac{n}{n-1})$, there exists a constant $C_S > 0$ such that for any $\vep \in (0,1]$,
\[
    \forall f \in L^2_1(\hX, \om_\vep),\quad
    \lt(\int_\hX |f|^{2q} \om_\vep^n\rt)^{1/q}
    \leq C_S \lt(\int_\hX |\dd f|_{\om_\vep}^2 \om_\vep^n + \int_\hX |f|^2 \om_\vep^n\rt).
\]
\end{bigthm}

In the localizing process, we choose K\"ahler domains that cover the resolution $\hX$; specifically, they are taken as the preimages of pseudoconvex domains covering $X$. 
Then using cutoff functions, we glue the local uniform Sobolev inequalities obtained through Guo--Phong--Song--Sturm's scheme into a global one for $(\om_\vep)_{\vep \in (0,1]}$. 

\smallskip
In our strategy, one requires solving complex Monge--Amp\`ere equations on the domains. 
Since the K\"ahler domains we select can intersect exceptional divisors of the resolution of singularities and the divisors can cross the boundary of these domains, one must handle non-smooth solutions of local complex Monge--Amp\`ere equations even when the right-hand side density is smooth. 
Recent results of complex Monge--Amp\`ere equations on singular pseudoconvex domains \cite{Guedj_Guenancia_Zeriahi_2023, Pan_2023, GGZ_2024} play a crucial role in overcoming such a difficulty. 

\smallskip
We remark that in \cite[Sec.~6]{Guo_Phong_Song_Sturm_2023_Sobolev}, the authors also obtain local uniform Sobolev inequalities on domains in a compact K\"ahler manifold $X$. 
They however require the global K\"ahler condition in the proof for defining Green's functions and using the estimates obtained from the global approach. 
This is not possible in our situation. 

\subsection*{Singular Hermite--Einstein metrics}
We then apply the existence of bounded Gauduchon metrics to study Hermite--Einstein metrics of stable reflexive sheaves on non-K\"ahler varieties.

\smallskip 
Recall that in the K\"ahler setting, Bando and Siu \cite{Bando_Siu_1994} extended Donaldson--Uhlenbeck--Yau theorem to stable reflexive sheaves over compact K\"ahler manifolds. 
This has been recently generalized to normal K\"ahler spaces \cite{Chen_Wentworth_2024, Chen_2022, Paun_et_al_2023_HE_sing}. 
Additionally, recent works have explored variational approaches in the smooth setting \cite{Hashimoto_Keller_2021, Jonsson_McCleerey_Shivaprasad_2022, Hashimoto_Keller_2023}. 

\smallskip
In the hermitian context, Li and Yau \cite{Li_Yau_1987} established Donaldson--Uhlenbeck--Yau's theorem on a compact non-K\"ahler manifold equipped with a Gauduchon metric (see also \cite{Buchdahl_1988} for the result in dimension two). 
Chen and Zhang \cite{Chen_Zhang_2024} recently extended this result to stable reflexive sheaves over compact non-K\"ahler manifolds via a heat flow method. 

\smallskip
The second aim of our work is to further extend the existence result to singular non-K\"ahler underlying spaces, and also to provide more precise estimates on singular Hermite--Einstein metrics as in the work of Cao--Graf--Naumann--P\u{a}un--Peternell--Wu~\cite{Paun_et_al_2023_HE_sing}. 
Additionally, we also obtain the uniqueness of these singular Hermite--Einstein metrics. 

\smallskip
Let $X$ be an $n$-dimensional compact variety, which, from now on, is assumed to be normal. 
Fix $\om_\RG$ a bounded Gauduchon metric on $X$. 
For a rank $r$ torsion-free coherent sheaf $\CF$ on $X$, by Rossi's theorem \cite{Rossi_1968}, there exists a log resolution of singularities $\pi: \hX \to X$ consisting of finitely many blowups with smooth centers such that $\pi^\sharp \CF := \pi^\ast \CF / \Tor(\pi^\ast \CF)$ is locally free (i.e. a rank $r$ holomorphic vector bundle). 
From \cite[Thm.~B]{Pan_2022} and its proof, $\om_\RG^{n-1}$ (resp. $\pi^\ast \om_\RG^{n-1}$) can be extended trivially as a pluriclosed current $T_\RG$ on the whole $X$ (resp. $\pi^\ast T_\RG$ on $\hX$). 
The $\om_\RG$-degree of $\CF$ is defined as 
\[
    \deg_{\om_\RG}(\CF) := \int_{\hX} \BCC(\pi^\sharp \CF) \w \pi^\ast T_\RG
\]
where $\BCC(\pi^\sharp\CF)$ is the first Bott--Chern class of $\pi^\sharp \CF$. 
The above definition does not depend on the resolution $\pi$ (see Section~\ref{sec:defn_stability} for more details). 
The $\om_\RG$-slope of $\CF$ is then defined by 
\[
    \mu_{\om_\RG}(\CF) := \frac{\deg_{\om_\RG}(\CF)}{\rk(\CF)}.
\]
We say that $\CF$ is $\mu_{\om_\RG}$-stable (resp. $\mu_{\om_\RG}$-semi-stable) if for every proper coherent subsheaf $\CG$ of $\CF$, 
\[
    \mu_{\om_\RG}(\CG) < \mu_{\om_\RG}(\CF) \quad \text{(resp. $\mu_{\om_\RG}(\CG) \leq \mu_{\om_\RG}(\CF)$)}.
\]

\smallskip
We now fix a coherent reflexive sheaf $\CE$ of rank $r$ on $X$.
Denote by $Z := X^\sing \cup \Sing(\CE)$ and $X^\circ := X \setminus Z$.
Fix a log resolution of singularities $\pi : \hX \to X$ such that $\pi_{|\pi^{-1}(X^\circ)}: \hX \setminus D \to X^\circ$ is an isomorphism where $D = \pi^{-1}(Z)$, and $E := \pi^\ast \CE/ \Tor(\pi^\ast \CE)$ is locally free; namely, $E$ is a rank $r$ holomorphic vector bundle. 
Recall that there is a smooth hermitian metric $\om$ on $X$ such that $\om_\RG = \rho \om$ and $1 \leq \rho \in L^\infty(X^\reg) \cap \CC^\infty(X^\reg)$. 
Also, there is a $\om$-plurisubharmonic function $\vph_Z \leq -1$ such that $\vph_Z$ is smooth on $X^\circ$, $Z = \{\vph_Z = -\infty\}$, and $\pi^\ast (\om + \ddc \vph_Z)$ dominates a hermitian metric on $\hX$. 
As shown in \cite[Thm.~1.1 and Prop.~2.2]{Paun_et_al_2023_HE_sing}, there exists a hermitian metric $h_0$ on $\CE$ in the sense of Grauert--Riemenschneider \cite{Grauert_Riemenschneider_1970} whose curvature satisfies the inequality
\[
    \ii \Ta(\CE, h_0) \geq -C \om \otimes \Id_{\CE}
\]
on $X^\circ$, for some constant $C > 0$. 
Moreover, $\pi^\ast (h_0)_{|X^\circ}$ extends to a smooth hermitian metric $h_E$ on $E$. 

\smallskip
Under the above setup, we obtain the following existence and uniqueness result of singular Hermite--Einstein metrics on slope-stable reflexive sheaves: 

\begin{bigthm}\label{bigthm:HE_metric}
If $\CE$ is $\om_\RG$-stable, then $\CE$ admits a singular $\om_\RG$-Hermite--Einstein metric $h^{\HE}$. 

More precisely, $h^{\HE} = h_0 H$ on $X^\circ$ where $H$ is a positive-definite endomorphism of $\CE_{|X^\circ}$, self-adjoint with respect to $h_0$. 
This metric satisfies the following properties: 
\begin{enumerate}
    \item 
    It is smooth and solves the Hermite--Einstein equation on $X^\circ$, i.e. 
    \[
        \tr_{\om_\RG} \ii\Ta(\CE, h^{\HE}) = \ld \Id_{\CE},
    \] 
    for a constant $\ld \in \BR$;
    \item 
    Write $H = e^{-f/r} S$ with the normalization $\Tr_{\End} \log S = 0$ and $\int_X f \om_\RG^n = 0$.
    There exists a constant $C > 0$ such that the following pointwise inequalities hold on $X^\circ$: 
    \[
        e^{C\vph_Z} \otimes \Id_\CE \leq S \leq e^{-C\vph_Z} \otimes \Id_\CE, 
        \quad -C \leq f \leq -C \vph_Z,
        \quad\text{and}\quad
        \Tr_{\End} H \leq C.
    \] 
    Moreover, one has  
    \[
        \int_{X^\circ} \frac{|D' \log H|_{h_0, \om_\RG}^2}{-\vph_Z} \om_\RG^n < + \infty,
    \]
    where $D'$ is the $(1,0)$-part of the Chern connection with respect to $h_0$ on $\CE_{|X^\circ}$.
\end{enumerate}
In addition, a singular Hermite--Einstein metric satisfying the above estimates and normalization is unique. 
\end{bigthm}

To obtain such a singular Hermite--Einstein metric, we proceed by approximation. 
We briefly outline the strategy here. 
By Theorem~\ref{bigthm:approx_and_conv}, we have a family of Gauduchon metrics $(\om_{\vep,\RG})_{\vep \in (0,1]}$ on $\hX$ which up to subsequence converges locally smoothly on $\hX \setminus D$ towards $\pi^\ast \om_\RG$ with their Gauduchon factors uniformly controlled. 
We next establish the $\om_{\vep,\RG}$-stability of $E$ for all sufficiently small $\vep$. 
By the theorem of Li and Yau \cite{Li_Yau_1987}, for all $\vep$ sufficiently small, there exists a $\om_{\vep,\RG}$-Hermite--Einstein metric $h_\vep^{\HE} = h_E H_\vep$ on $E$. 
Following a scheme similar to \cite[Sec.~3-4]{Paun_et_al_2023_HE_sing}, we derive uniform estimates for $H_\vep$ ensuring the existence of a converging subsequence in $\CC^\infty_{\loc}(\hX \setminus D)$, with the limit being the desired singular Hermite--Einstein metric. 

\smallskip
We highlight below two main differences with the K\"ahler setting: 

\smallskip
\begin{itemize}
\item 
{\bf Openness of slope stability.} 
To establish the $\om_{\vep,\RG}$-stability of $E$, we follow a contradiction argument in \cite{Bruasse_2001, Chen_Zhang_2024}.
In the K\"ahler setting, the openness of stability is more straightforward due to the linear variation of classes and the well-established theory on finiteness results of the Douady space (see e.g. \cite[Lem.~3.3, Lem.~3.4]{Claudon_Graf_Guenancia_2022} and \cite[Cor.~6.3]{Toma_2021}). 
However, in our case, the metrics $(\om_{\vep,\RG})_{\vep \in (0,1]}$ and their Aeppli cohomology classes evolve in a non-linear way, and there is no well-established theory of stable sheaves in our setting. 

\smallskip
\item 
{\bf Uniqueness of Hermite--Einstein metrics.} 
Using the estimates provided in Theorem~\ref{bigthm:HE_metric}, we establish the uniqueness of singular Hermite--Einstein metrics. 
In the non-K\"ahler setting, even when the underlying space $X$ is smooth, the uniqueness of such metrics on stable reflexive sheaves that are not locally free is a novel result. 
In the K\"ahler context, uniqueness is proved by using the "admissible condition" (i.e. finite $L^2$-norm of the curvature) of singular Hermite--Einstein metrics (cf. \cite{Bando_Siu_1994, Chen_Wentworth_2024, Chen_2022}). 
However, in the hermitian case, the $L^2$-norm of the curvature cannot be directly controlled by topological quantities when $n \geq 3$, so it remains unclear whether these singular Hermite--Einstein metrics are admissible. 
As a result, the estimates in Theorem~\ref{bigthm:HE_metric} become crucial to our approach to proving uniqueness.
\end{itemize}

\smallskip
We also stress that the key Harnack-type estimate used in \cite[Thm.~2.7]{Paun_et_al_2023_HE_sing} was originally established in \cite[Lem.~2]{Guo_Phong_Sturm_2024} via a global approach involving auxiliary Monge--Amp\`ere equations. 
In our situation, as reasons mentioned before, this global method is not well-suited. 
Hence, we establish a Harnack-type estimate via the more standard approach of Moser iteration, and the uniform Sobolev inequality for the desired family of metrics is available, thanks to Theorem~\ref{bigthm:unif_Sobolev}. 

\smallskip
Furthermore, the application of uniform Sobolev and Poincar\'e inequalities allows us to slightly simplify the approach in \cite{Paun_et_al_2023_HE_sing} on the control of $\Tr_{\End} \log H_\vep$. 

\subsection*{Organization of the article}
\begin{itemize}
    \item Section~\ref{sec:prelim} recalls basic notions of quasi-subharmonic and quasi-plurisubharmonic functions. 
    
    \item Section~\ref{sec:unif_Sobolev} establishes uniform Sobolev inequalities on K-domains. 
    
    \item Section~\ref{sec:SGM} combines uniform Sobolev inequalities in Section~\ref{sec:unif_Sobolev} with the geometry of a resolution of singularities to prove Theorem~\ref{bigthm:unif_Sobolev} and Theorem~\ref{bigthm:approx_and_conv}. 
    
    \item Section~\ref{sec:slope_openness} reviews the concept of slope-stability and studies an openness property of stability on resolutions. 
    
    \item Section~\ref{sec:HE_estimates} focuses on proving Theorem~\ref{bigthm:HE_metric}. 
\end{itemize}

\begin{ack*}
The author is grateful to Vincent~Guedj, Henri~Guenancia, Mihai~P\u{a}un, and Song~Sun for helpful discussions and suggestions. 
The author is indebted to Tam\'as~Darvas and Junsheng~Zhang for their comments and to Julien~Keller, Yangyang~Li, Davide~Parise, and Junsheng~Zhang for useful references. 
The author would also like to thank Tristan Collins for discussing related problems during the author's visit to MIT. 

\smallskip
Part of this article is based on work supported by the National Science Foundation under Grant No. DMS-1928930, while the author was in residence at the Simons Laufer Mathematical Sciences Institute (formerly MSRI) in Berkeley, California, during the Fall 2024 semester. 
The author is also supported by postdoc funding from CIRGET, UQAM.
\end{ack*}

\section{Preliminaries}\label{sec:prelim}
In this section, we review the notions of subharmonic and plurisubharmonic functions that will be frequently used to prove main estimates. 
We denote by $\dc = \frac{\ii}{2\pi}(\db - \pl)$ the twisted exterior derivative so that $\ddc = \frac{\ii}{\pi} \ddb$.
Throughout the article, the term "variety" always refers to an irreducible reduced complex analytic space. 

\subsection{Subharmonic functions}
Let $U$ be an open set in an $n$-dimensional complex manifold $X$ and let $\gm$ be a hermitian positive $(1,1)$-form on $U$. 
A function $h \in \CC^2(U)$ is called {\it $\gm$-harmonic} if 
\[
    \ddc h \w \gm^{n-1} = 0 \quad \text{in} \ \, U.
\]
Denote by $\Har_\gm(U)$ the set of all $\gm$-harmonic functions on $U$. 

\begin{defn}
A function $u: U \to \BR \cup \{-\infty\}$ is said to be {\it $\gm$-subharmonic} if $u$ is upper semi-continuous and for every $D \subset U$ and every $h \in \Har_\gm(D)$ the following implication holds: 
\[
    u \leq h \ \, \text{on} \ \, \pl D
    \implies
    u \leq h \ \, \text{in} \ \, D.
\]
Set $\SH_\gm(U)$ the set of all $\gm$-subharmonic functions on $U$ which are locally integrable.
\end{defn}

The notion of the $\gm$-subharmonicity of a function $u$ is essentially equivalent to the inequality $\ddc u \w \gm^{n-1} \geq 0$ in the sense of distributions. 
More precisely, if $u \in \SH_\gm(U)$, by \cite[Thm.~1, p.~136]{Herve_Herve_1972}, $\ddc u \w \gm^{n-1} \geq 0$ in the sense of distributions. 
Conversely, \cite[Thm.~A]{Littman_1963} shows that the later condition implies that $u$ coincides almost everywhere with a unique function $u^\ast \in \SH_\gm(U)$ given by 
\[
    u^\ast(x) := \esslimsup_{y \to x} u(y) := \lim_{r \to 0} \esssup_{\BB_r(x)} u
\]
where $\esssup_{\BB_r(x)} u$ is the essential supremum of $u$ on $\BB_r(x)$ the ball of radius $r$ centered at $x$.

\subsection{Plurisubharmonic functions and forms on singular spaces}
Let $\Om$ be a domain in $\BC^n$. 
A {\it plurisubharmonic} (psh) function $\vph: \Om \to \BR \cup \{-\infty\}$ is an upper semi-continuous function satisfying the sub-mean value inequality along every complex line, i.e. for all $x \in \Om$, and all $\xi \in \BC^n$ so that $|\xi| < \dist (x, \pl\Om)$, 
\[
    \vph(x) \leq \frac{1}{2\pi} \int_0^{2\pi} \vph(x + \xi e^{\ii \ta}) \dd \ta. 
\]
Since the plurisubharmonic condition is preserved under composition with a holomorphic map, the plurisubharmonicity still makes sense on complex manifolds. 
We shall denote by $\PSH(\Om)$ the set of all psh functions that are not identically $-\infty$ on $\Om$, which is a domain in $\BC^n$ or in a complex manifold. 

\smallskip
More general notions of psh functions have been defined on complex manifolds: 
\begin{defn} 
Let $X$ be a complex manifold and let $\om$ be a smooth real $(1,1)$-form on $X$. 
\begin{enumerate}
    \item A function $\vph: X \to \BR \cup \{-\infty\}$ is {\it quasi-plurisubharmonic} (quasi-psh) if locally $\vph$ can be decomposed as a sum of a smooth function and a psh function. 
    \item A function $\vph$ is {\it $\om$-plurisubharmonic} ($\om$-psh) if $\vph$ is quasi-psh and $\om + \ddc \vph \geq 0$ in the sense currents.
\end{enumerate}
\end{defn}

Now, suppose that $X$ is a reduced complex analytic space of pure dimension $n \geq 1$. 
We denote by $X^\reg$ the complex manifold of regular points of $X$ and $X^\sing = X \setminus X^\reg$ the singular set of $X$. 
The notions of smooth forms/metrics/currents/psh functions can also be extended to singular contexts. 
We quickly review the basic definitions here (see \cite{Demailly_1985} for further details): 

\begin{defn}
We say that
\begin{enumerate}
    \item $\ta$ is a smooth form on $X$ if it is smooth form on $X^\reg$ and given any local embedding $X \xhookrightarrow[\loc.]{} \BC^N$, $\ta$ extends smoothly to $\BC^N$.
    
    \item $\om$ is a smooth hermitian metric on $X$ if it is a smooth $(1,1)$-form which locally extends to a hermitian metric on $\BC^N$. 
\end{enumerate}
The notion of currents is then defined by duality.
\end{defn}

Psh functions on $X$ are locally the restriction of psh functions under local embeddings of $X$ into $\BC^N$.
One defines similarly quasi-psh and $\om$-psh functions.
We denote by $\PSH(X,\om)$ the set of all integrable $\om$-psh functions on $X$.

\smallskip
Forn{\ae}ss--Narasimhan \cite[Thm.~5.3.1]{FN_1980} proved that the plurisubharmonicity of a function $u$ is equivalent to the following: for any analytic disc $h: \BD \to X$, $u \circ h$ is either subharmonic on $\BD$ or identically $-\infty$. 
If $u$ is a psh function on $X^\reg$ and locally bounded from above on $X$, one can extend $u$ to $X$ as follows: 
\begin{equation}\label{eq:usc_ext}
    u^\ast(x) = \limsup_{X^\reg \ni y \to x} u(y).
\end{equation}
When $X$ is locally irreducible, Grauert and Remmert \cite[Satz~3]{Grauert_Remmert_1956} have shown that if $u$ is psh on $X^\reg$, then the function $u^\ast$ defined by \eqref{eq:usc_ext} is psh on $X$ (see also \cite[Thm.~1.7]{Demailly_1985}).

\smallskip
We now further assume that $X$ is locally irreducible. Following \cite{FN_1980}, we say that $X$ is Stein if it admits a $\CC^2$ strongly psh exhaustion. 
A domain $\Om \Subset X$ is strongly pseudoconvex if it admits a negative $\CC^2$ psh exhaustion; namely, there is a $\CC^2$ psh function $\rho$ defined near a neighborhood $\Om'$ of $\overline{\Om}$ such that $\Om = \set{x \in \Om'}{\rho(x) < 0}$ and $\Om_c := \set{x \in \Om'}{\rho(x) < c} \Subset \Om$ is relatively compact for any $c < 0$. 

\smallskip
We also recall the notion of singular Gauduchon metrics introduced in \cite[Def.~1.2]{Pan_2022}:

\begin{defn}\label{defn:bounded_Gauduchon}
A hermitian metric $\om_\RG$ on a complex variety $X$ is 
\begin{enumerate}
    \item {\it Gauduchon} if it satisfies $\ddc \om_\RG^{n-1} = 0$ on $X^\reg$; 
    \item {\it bounded Gauduchon} if there exists a smooth hermitian metric $\om$ on $X$ and a positive bounded smooth function $\rho$ on $X^\reg$ such that $\om_\RG = \rho^{\frac{1}{n-1}} \om$ and $\ddc (\rho \om^{n-1}) = 0$ on $X^\reg$.
\end{enumerate}
\end{defn}

\section{Uniform Sobolev inequality on K-domains}\label{sec:unif_Sobolev}
This section aims to prove uniform Sobolev inequalities on K-domains. 
These are domains where one can solve complex Monge--Amp\`ere equations and obtain a Ko{\l}odziej-type $L^\infty$-estimate. 

\begin{defn}\label{defn:K_domain}
Let $\Om$ be an $n$-dimensional complex manifold with a smooth boundary, and let $\dd V$ be a smooth volume form on $\Om$ up to the boundary. 
We call $\Om$ a \emph{K-domain} if for any $p>1$ and for any $f \in L^p(\Om, \dd V)$, the complex Monge--Amp\`ere equation 
\[
    (\ddc \vph)^n = f \dd V 
    \quad\text{with}\quad
    \vph_{|\pl\Om} = 0
\]
admits a solution $\vph \in \PSH(\Om) \cap \CC^0(\overline{\Om})$ and there is a uniform constant $C_{\Om,p} > 0$ such that 
\[
    \|\vph\|_{L^\infty(\Om)} \leq C_{\Om,p} \|f\|_{L^p(\dd V)}^{1/n}.
\]
\end{defn}

By Ko{\l}odziej's $L^\infty$-estimate \cite{Kolodziej_1998}, strongly pseudoconvex domains in $\BC^n$ are examples of K-domains. 
In the next section, we shall prove that generic desingularizations of strongly pseudoconvex domains in Stein spaces are also K-domains (see Lemma~\ref{lem:resoln_pscvx_Kdomain}).

\begin{defn}
Let $\Om$ be an $n$-dimensional complex manifold with a smooth boundary, and let $\dd V$ be a smooth volume form on $\Om$, extending smoothly up to the boundary.  
For $p > 1$ and $A > 0$, we define the set $\CK(\Om, \dd V, p, A)$ as the collection of K\"ahler metrics, where each $\om$ extends to an open neighborhood $U_\om$ of $\overline{\Om}$, and satisfies the following integrability condition:
\[
    \lt(\int_\Om f_\om^p \, \dd V\rt)^{1/p} \leq A,
    \quad \text{where $f_\om := \frac{\om^n}{\dd V}$.}
\]
\end{defn}

\subsection{Estimates of Green's function on K-domains}
Let $(M,g)$ be a compact oriented Riemannian manifold with boundary $\pl M$, and let $\Dt_g$ be the Laplace--Beltrami operator on $(M,g)$ acting on functions with zero boundary data. 
From the standard elliptic theory, there exists a Green's function $G$ defined on $(M \times M) \setminus \set{(x,x)}{x \in M}$ such that
\[
    \int_M G(x,y) \Dt_g f(y) \, \dd \mathrm{vol}_g = f(x) 
\]
for all $f \in \CC^\infty(\overline{\Om})$ satisfying $f_{|\pl M} = 0$; namely, for a fixed $x \in M$, $\Dt_g G(x,y) = \dt_x$.
Moreover, for any $x \in M$ and $y \in \pl M$, $G(x,y) = 0$.
In the sequel, we shall denote by $G_\om$ the Green's function associated with a K\"ahler metric $\om$,
\[
    \Dt_\om f := n\frac{\ddc f \w \om^{n-1}}{\om^n} 
    \quad\text{and}\quad
    |\dd f|_\om^2 := n \frac{\dd f \w \dc f \w \om^{n-1}}{\om^n}.
\]

\smallskip
In \cite{Guo_Phong_Song_Sturm_2024_diameter}, Guo--Phong--Song--Sturm have developed several estimates on Green's functions on compact K\"ahler manifolds.  
We shall follow the general scheme of \cite{Guo_Phong_Song_Sturm_2024_diameter} to establish estimates for Green's functions on K-domains. 

\smallskip
We recall a pluripotential arithmetic and geometric mean inequality due to Ko{\l}odziej \cite[Lem.~1.2]{Kolodziej_2003} and a more general version later by Dinew \cite[Thm.~3]{Dinew_2009}. 
The following version is a simple version for our later usage: 

\begin{lem}\label{lem:AMGM_mes}
Let $u$ and $v$ be bounded psh functions in a domain of $\BC^n$. 
Suppose that 
\[
    (\ddc u)^n \geq f \dd V
    \quad\text{and}\quad
    (\ddc v)^n \geq g \dd V
\]
for $f, g \in L^1(\dd V)$, where $\dd V$ is a smooth volume form. 
Then for any $k \in \{1, \cdots, n-1\}$, 
\[
    (\ddc u)^k \w (\ddc v)^{n-k} \geq f^{\frac{k}{n}} g^{\frac{n-k}{n}} \dd V.
\]
\end{lem}

We are going to show that $\om$-subharmonic functions are uniformly bounded from below when their Laplacian measures have a uniform $L^{nl}$ control for $l > \frac{p}{p-1}$. 

\begin{lem}\label{lem:subhar_est}
Fix $p > 1$ and $A > 0$. 
Let $\Om$ be a K-domain, $\dd V$ be a smooth volume form on $\Om$, and $\om \in \CK(\Om, \dd V, p, A)$.
Suppose that $v \in \SH_\om(\Om)$ satisfies $v_{|\pl\Om} = 0$ and
\[
    \ddc v \w \om^{n-1} \leq g \om^n 
\]
for some $g \in L^{nl}(\dd V)$ with $l \in \lt(\max\lt\{\frac{p}{p-1}, 2\rt\}, +\infty\rt]$.
Then 
\[
    \|v\|_{L^\infty(\Om)} \leq C_{\Om,r} \cdot A^{1/n} \cdot \|g\|_{L^{nl}(\dd V)}, 
\] 
where $r > 1$ is a constant such that $1/r = 1/p + 1/l$.
\end{lem}

\begin{proof}
As $g \in L^{nl}(\dd V)$ for some $l > 2$, $v$ is continuous by the standard elliptic theory (cf. \cite[Thm.~9.30]{Gilbarg_Trudinger}). 
Since $\Om$ is a K-domain, there exists a solution $u \in \PSH(\Om) \cap \CC^0(\overline{\Om})$ to 
\[
    (\ddc u)^n = g^n \om^n 
    \quad\text{with}\quad
    u_{|\pl\Om} = 0.
\]
By Lemma~\ref{lem:AMGM_mes} and the assumption on $v$, we have 
\[
    \ddc u \w \om^{n-1} 
    \geq g \om^n 
    \geq \ddc v \w \om^{n-1}.
\]
Thus, $(u-v)$ is also $\om$-subharmonic, as $(u-v)$ is continuous.  
By zero boundary values of $u$ and $v$, and the maximum principle, we obtain $(u-v) \leq 0$ on $\Om$, and thus, $u \leq v \leq 0$. 
Set $r > 1$ such that $1/r = 1/p + 1/l$.  
By Ko{\l}odziej's estimate on $\Om$ and H\"older inequality, one can derive the following
\begin{align*}
    \|v\|_{L^\infty(\Om)} 
    &\leq \|u\|_{L^\infty(\Om)} 
    \leq C_{\Om,r} \|g^n \cdot f_\om\|_{L^r(\dd V)}^{1/n}\\
    &\leq C_{\Om,r} \cdot \|f_\om\|_{L^p(\dd V)}^{1/n} \|g\|_{L^{nl}(\dd V)} 
    \leq C_{\Om,r} \cdot A^{1/n} \cdot \|g\|_{L^{nl}(\dd V)}.
\end{align*}
\end{proof}

Using Lemma~\ref{lem:subhar_est}, we establish the following uniform $L^1$-estimate of Green's functions: 

\begin{lem}\label{lem:G_L1_est}
Fix $p > 1$ and $A > 0$.
Let $\Om$ be a K-domain, and $\dd V$ be a smooth volume form on $\Om$. 
There exists a constant $C_0 := \frac{C_{\Om,p} \cdot A^{1/n}}{n}$ such that for every $x \in \Om$, and for every $\om \in \CK(\Om, \dd V, p, A)$, 
\[
    \int_{\Om} -G_\om(x,\cdot) \om^n \leq C_0.
\]
\end{lem}

\begin{proof}
Let $(\chi_k)_k$ be a sequence of smooth functions with compact support in $\Om$ and $\chi_k$ increases towards $\1_{\Om}$ the characteristic function of $\Om$.
Consider $v_k$ a solution to the following Poisson equation with zero Dirichlet boundary condition:
\[
    \Dt_{\om} v_k = \chi_k
    \quad\text{and}\quad
    {v_k}_{|\pl\Om} = 0.
\]
By Lemma~\ref{lem:subhar_est}, for all $k$, we have $\|v_k\|_{L^\infty(\Om)} \leq C_0 \cdot V_\om(\Om)^{1/n}$ where $0 < C_0 = \frac{C_{\Om,p} \cdot A^{1/n}}{n}$ is a uniform constant.
By the Green formula and monotone convergence theorem, we have
\[
    v_k(x) = \int_{\Om} G_\om(x,\cdot) \Dt_{\om} v_k \om^n
    = \int_{\Om} G_\om(x,\cdot) \chi_k \om^n 
    \xrightarrow[k \to +\infty]{} 
    \int_{\Om} G_\om(x,\cdot) \om^n.
\]
Using the uniform estimate on $\|v_k\|_{L^\infty(\Om)} \leq C_0$, we complete the proof.
\end{proof}

Next, we provide a uniform $L^{\af}$-estimate on Green's functions for $\af \in \lt(1, \frac{n}{n-1}\rt)$:

\begin{lem}\label{lem:G_Lp_est}
Fix $p > 1$ and $A > 0$.
Let $\Om$ be a K-domain and $\dd V$ be a smooth volume form on $\Om$. 
For any $\af \in (1, \frac{n}{n-1})$, there exists a constant $C_1 > 0$ depending only on $n, p, \af, C_{\Om,p}, A, \Om$ such that for all $\om \in \CK(\Om, \dd V, p, A)$ and for all $x \in \Om$, 
\[
    \int_\Om (-G_\om(x,\cdot))^\af \om^n \leq C_1.
\]
\end{lem}

\begin{proof}
Set $\CG_\om(x,\cdot) := -G_\om(x,\cdot)$ and we shall simply denote it by $\CG_\om$ later.
Following the strategy in \cite[Thm.~2.3]{Guedj_To_2024}, we are going to show that if $\int_\Om \CG_\om^{n\bt} \om^n \leq C_1$ for $n \bt \leq \bt'$ then $\int_\Om \CG_\om^\af \om^n \leq C(\af, A, \Om, C_1)$ for $\af < 1 + \frac{\bt'}{n}$.

\smallskip
For $L \geq 1$, consider the following equations with zero Dirichlet boundary conditions: 
\[
    \ddc u \w \om^{n-1} = \CG_{\om,L}^\dt \om^n
    \quad\text{with}\quad
    u_{|\pl \Om} = 0
\]
and 
\[
    (\ddc \vph)^n = \CG_{\om,L}^{n \dt} \om^n
    \quad\text{with}\quad
    \vph_{|\pl \Om} = 0
\]
where $\CG_{\om,L} = \min\{\CG_\om, L\}$ and $\dt$ is a constant to be determined.
Fix $r \in (1,p)$. 
By H\"older inequality, we have 
{\small 
\[
    \int_\Om (\CG_{\om,L}^{n\dt} f_\om)^r \dd V
    = \int_\Om \CG_{\om,L}^{n \dt r} f_\om^{r-1} \om^n 
    \leq \lt(\int_\Om \CG_\om^{n \dt r p'} \om^n\rt)^{1/p'} \lt(\int_\Om f_\om^{p} \dd V \rt)^{\frac{r-1}{p-1}}
    \leq A^{\frac{p(r-1)}{p-1}} \cdot \lt(\int_\Om \CG_\om^{n \dt r p'} \om^n\rt)^{1/p'}
\]
}%
where $p' = \frac{p-1}{p-r}$. 
Put $\dt = \frac{\bt'}{nrp'}$.
Then one obtains 
\[
    \int_\Om (\CG_{\om,L}^{n\dt} f_\om)^r \dd V 
    \leq A^{\frac{p(r-1)}{p-1}} \cdot C_1^{1/p'}.
\]
From the condition of K-domain, we get
\[
    \|\vph\|_{L^\infty(\Om)} 
    \leq C_{\Om,r} \|\CG_{\om,L}^{n\dt} f_\om\|_{L^r(\dd V)}^{1/n}
    \leq C_{\Om,r} \cdot 
    A^{\frac{p(r-1)}{nr(p-1)}} \cdot C_1^{\frac{1}{nrp'}}.
\]
As in the proof of Lemma~\ref{lem:subhar_est}, we have $\|u\|_{L^\infty(\Om)} \leq \|\vph\|_{L^\infty(\Om)}$; hence, 
\[
    C_{\Om,r} \cdot 
    A^{\frac{p(r-1)}{nr(p-1)}} \cdot C_1^{\frac{1}{nrp'}} 
    \geq -u(x) 
    = \int_\Om -G_\om(x,\cdot) \Dt_\om u \om^n 
    = n \int_\Om -G_\om(x,\cdot) \CG_{\om,L}^\dt \om^n 
    \xrightarrow[L \to +\infty]{} n \int_\Om \CG_\om^{1+\dt} \om^n.
\]
For $r \to 1$, one has $p' \to 1$, and one can conclude that, for any $\af < 1 + \frac{\bt'}{n}$, letting $r \in (1,p)$ so that $\af = 1 + \frac{\bt'(p-r)}{nr(p-1)}$, 
\[
    \int_\Om \CG^{\af}_\om \om^n 
    \leq C_{\Om,r} \cdot A^{\frac{p(r-1)}{nr(p-1)}} \cdot C_1^{\frac{(p-r)}{nr(p-1)}}.
\]

\smallskip
From Lemma~\ref{lem:G_L1_est}, starting with $\bt' = 1$, iterating the above argument, we obtain that, for any $\af < \frac{n}{n-1} = \sum_{i=0}^\infty \frac{1}{n^i}$, there exists a constant $C_1$ depending only on $\af$, $A$, $\Om$, $C_{\Om,p}$ such that for all $\om \in K(\Om, \dd V, p, A)$ and for all $x \in \Om$, 
\[
    \int_{\Om} (-G_{\om}(x,\cdot))^\af \om^n \leq C_1 
\]
and this finishes the proof. 
\end{proof}

We next provide a uniformly weighted gradient estimate for Green's functions: 

\begin{lem}\label{lem:G_grad_est}
Let $\om$ be a K\"ahler metric on $\Om$. 
For all $\bt > 0$ and for any $x \in \Om$, 
\[
    \int_{\Om} \frac{|\dd G_\om(x,\cdot)|^2_{\om}}{(-G_\om(x,\cdot) + 1)^{1+\bt}} \om^n 
    = \frac{1}{\bt}.
\]
\end{lem}

\begin{proof}
Fix $x \in \Om$ and $\bt > 0$. 
Take $u(y) = (-G_\om(x,y) + 1)^{-\bt} - 1$ which is smooth on $\Om \setminus \{x\}$. 
Note that $u(x) = -1$ and $u_{|\pl \Om} = 0$. 
By Stokes' formula, we have the following
\[
    -1 = u(x) = \int_{\Om} G_\om(x,\cdot) \Dt_{\om} u \om^n
    = -\bt \int_{\Om} \frac{|\dd G_\om(x,\cdot)|^2_\om}{(-G_\om(x,\cdot) + 1)^{1+\bt}} \om^n.
\]
\end{proof}

\subsection{Uniform Sobolev inequality on K-domains}

Based on the estimates built in the previous section, we establish the following uniform Sobolev inequality: 

\begin{thm}\label{thm:loc_Sobolev}
Fix $p > 1$ and $A > 0$.
Let $\Om$ be a K-domain and $\dd V$ be a smooth volume form on $\Om$. 
Then for any $q \in (1,\frac{n}{n-1})$, there is a constant $C_S > 0$, such that for all $\om \in \CK(\Om, \dd V, p, A)$, and for every $u \in \CC^\infty_c(\Om)$, 
\[
    \lt(\int_{\Om} |u|^{2q} \om^n\rt)^{1/q} 
    \leq C_S \int_{\Om} |\dd u|_{\om}^2 \om^n.
\]
\end{thm}

\begin{proof}
We follow the method in \cite[Sec.~5]{Guo_Phong_Song_Sturm_2023_Sobolev} to conclude. 
For any $u \in \CC^\infty_c(\Om)$, by Stokes' formula,
\[
    u(x) = \int_{\Om} G_\om(x, \cdot) \Dt_\om u \om^n
    = n \int_{\Om} \dd (-G_\om(x, \cdot)) \w \dc u \w  \om^{n-1}.
\]
Using H\"older inequality and Lemma~\ref{lem:G_grad_est}, we have 
\begin{align*}
    |u(x)| 
    &\leq \int_{\Om} |\dd G_\om(x,\cdot)|_\om |\dd u|_\om \om^n
    = \int_{\Om} \frac{|\dd G_\om(x,\cdot)|_\om}{(-G_\om(x,\cdot)+1)^{(1+\bt)/2}} \cdot (-G_\om(x,\cdot)+1)^{(1+\bt)/2} |\dd u|_\om \om^n\\
    &\leq \lt(\int_{\Om} \frac{|\dd G_\om(x,\cdot)|_\om^2}{(-G_\om(x,\cdot)+1)^{(1+\bt)}} \om^n\rt)^{1/2}
    \lt(\int_{\Om} (-G_\om(x,\cdot)+1)^{(1+\bt)} |\dd u|_\om^2 \om^n\rt)^{1/2}\\
    &\leq \lt(\frac{1}{\bt}\rt)^{1/2} \lt(\int_{\Om} (-G_\om(x,\cdot)+1)^{(1+\bt)} |\dd u|_\om^2 \om^n\rt)^{1/2}.
\end{align*}
Fix $\bt \in (0,1)$ such that $(1+\bt) q < \frac{n}{n-1}$. 
Then 
\[
    |u(x)|^{2q} \leq \frac{1}{\bt^q} \cdot \lt(\int_\Om (-G_\om(x,\cdot)+1)^{(1+\bt)} |\dd u|_\om^2 \om^n\rt)^{q},
\]
and thus, by the generalized Minkowski inequality,
\begin{equation}\label{eq:generalized_Minkowski}
\begin{split}
    \lt(\int_{\Om} |u|^{2q} \om^n\rt)^{1/q} 
    &\leq \frac{1}{\bt} \lt\{\int_{x \in \Om} \lt(\int_{y \in \Om} (-G_\om(x,y)+1)^{(1+\bt)} |\dd u(y)|_\om^2 \om^n(y)\rt)^{q} \om^n(x) \rt\}^{1/q}\\
    &\leq \frac{1}{\bt} \int_{y \in \Om} |\dd u(y)|_\om^2 \lt(\int_{x \in \Om} (-G_\om(x,y)+1)^{(1+\bt)q} \om^n(x)\rt)^{1/q} \om^n(y).
\end{split}    
\end{equation}
Recall that from Lemma~\ref{lem:G_Lp_est}, for any $\af \in (0, \frac{n}{n-1})$, there exists a constant $C_1' > 0$, which does not depend on $\om \in K(\Om, \dd V, p, A)$ and $x \in \Om$, such that
\begin{equation}\label{eq:thm_loc_Sob_G_Lp}
    \int_{\Om} (-G_\om(x,\cdot) + 1)^{\af} \om^n \leq C_1'. 
\end{equation}
Combining \eqref{eq:generalized_Minkowski} and \eqref{eq:thm_loc_Sob_G_Lp}, we obtain the following uniform Sobolev inequality as desired
\begin{align*}
    \lt(\int_{\Om} |u|^{2q} \om^n\rt)^{1/q} 
    &\leq \frac{(C_1')^{1/q}}{\bt} \cdot \int_{\Om} |\dd u|_\om^2 \om^n.
\end{align*}
\end{proof}

\subsection{Uniform Sobolev inequality on compact hermitian manifolds}
Consider now $X$ an $n$-dimensional compact complex manifold, and $\dd V$ a smooth volume on $X$. 
Let $(\Om_i)_{i \in I}$ be a finite collection of K-domains covering $X$, and let $(\chi_i)_{i \in I}$ to be a partition of unity subordinating to $(\Om_i)_i$. 

\begin{defn}
Let $\CH((\Om_i)_i, (\chi_i)_i, \dd V, p, A_1, A_2, p', A_3)$ be a collection of hermitian metrics $\om$ on $X$, satisfying the following: 
\begin{itemize}
    \item 
    for each $i$, there exists $\om_{i,K} \in \CK(\Om_i, \dd V, p, A_1)$ such that $A_2^{-1} \om_{i,K} \leq \om \leq A_2 \om_{i,K}$;

    \item 
    $p' > \frac{q}{q-1}$ and $\|\dd \chi_i\|_{L^{2p'}(\Om_i, \om)} \leq A_3$ for all $i \in I$, where $q \in \lt(1, \frac{n}{n-1}\rt)$.
\end{itemize}
\end{defn}

\begin{thm}\label{thm:global_Sobolev}
Let $X$ be a compact complex manifold, $(\Om_i)_{i \in I}$ be a finite collection of K-domains covering $X$, and $\dd V$ be a smooth volume form on $X$ and $(\chi_i)_{i \in I}$ to be a partition of unity subordinating to $(\Om_i)_i$. 
For any $q \in \lt(1,\frac{n}{n-1}\rt)$, there exists $C_S>0$ such that for all $\om \in \CH((\Om_i)_i, (\chi_i)_i, \dd V, p, A_1, A_2, p', A_3)$, 
\[
    \forall u \in L^2_1(X,\om),
    \quad
    \lt(\int_X |u|^{2q} \om^n\rt)^{1/q} 
    \leq C_S \lt(\int_X |\dd u|_\om^2 \om^n + \int_X |u|^2 \om^n\rt).
\]
\end{thm}

\begin{proof}
Fix $q \in \lt(1,\frac{n}{n-1}\rt)$ and $p' > \frac{q}{q-1}$. 
For $\om \in \CH((\Om_i)_i, \dd V, p, A_1, A_2, p', A_3)$, for each $i$, there is $\om_{i,K} \in \CK(\Om_i, \dd V, p, A_1)$ such that $A_2^{-1} \om_{i,K} \leq \om \leq A_2 \om_{i,K}$. 
By H\"older and interpolation inequalities, one can derive the following
{\small
\begin{align*}
    &A_2^{-(n-1)} \int_{\Om_i} |\dd (\chi_i u)|_{\om_{i,K}}^2 \om_{i,K}^n
    \leq \int_{\Om_i} |\dd (\chi_i u)|_\om^2 \om^n 
    \leq 2 \int_{\Om_i} |\dd u|_\om^2 \om^n + 2 \int_{\Om_i} |\dd \chi_i|_\om^2 |u|^2 \om^n\\
    &\leq 2\int_{\Om_i} |\dd u|_\om^2 \om^n + 2\|\dd \chi_i\|_{L^{2p'}(\Om_i,\om)}^2  \|u\|_{L^{2q'}(\Om_i, \om^n)}^2 
    \leq 2 \int_{\Om_i} |\dd u|_\om^2 \om^n + 4 A_3 
    \lt(\vep^2 \|u\|_{L^{2q}(\Om_i, \om)}^2 + \vep^{-2\bt} \|u\|_{L^2(\Om_i, \om)}^2\rt)
\end{align*}
}%
where $1/p' + 1/q' = 1$ and $\bt = (\frac{1}{2} - \frac{1}{2q'})/(\frac{1}{2q'} - \frac{1}{2q})$ and all $\vep > 0$.
From Theorem~\ref{thm:loc_Sobolev}, we derive for some uniform $C = C(|I|, A_2, q) > 0$
\begin{align*}
    \lt(\int_X |u|^{2q} \om^n\rt)^{1/q} 
    &\leq C \sum_{i \in I} \lt(\int_{\Om_i} |\chi_i u|^{2q} \om_{i,K}^n\rt)^{1/q}
    \leq C \sum_{i \in I} \lt(\int_{\Om_i} |\dd (\chi_i u)|_{\om_{i,K}}^2 \om_{i,K}^n\rt)\\ 
    &\leq C'_S \lt[\int_X |\dd u|_\om^2 \om^n + \vep^{-2\bt} \int_X |u|^2 \om^n + \vep^2 \lt(\int_X |u|^{2q} \om^n\rt)^{1/q}\rt],
\end{align*}
where $C_S' = C_S'(|I|, A_2, q, A_3) > 0$. 
Then one can complete the proof by taking $\vep^2 < 1/C_S'$. 
\end{proof}

\section{Singular Gauduchon metrics}\label{sec:SGM}
This section aims to prove the existence result of bounded Gauduchon metrics and establish uniform a priori estimates for its approximations (cf. Theorem~\ref{bigthm:approx_and_conv}). 
A key to building a priori estimates is a uniform Sobolev inequality obtained in the previous section. 

\subsection{Uniform Sobolev and Poincar\'e inequalities}
We shall show uniform Sobolev and Poincar\'e inequalities along a perturbed family of hermitian metrics on a resolution of a singular variety in this section. 

\smallskip
We recall a powerful $L^\infty$-estimate of solutions to complex Monge--Amp\`ere equations, originally established by Ko{\l}odziej \cite{Kolodziej_1998} on strongly pseudoconvex domains in $\BC^n$.
The following refined version of the existence result and $L^\infty$-estimate can be found in the combination of \cite[Thm.~A]{Guedj_Guenancia_Zeriahi_2023}, \cite[Thm.~3.1]{Pan_2023} and \cite[Thm.~3.6]{GGZ_2024}. 

\begin{thm}\label{thm:Kolodziej}
Let $X$ be an $n$-dimensional locally irreducible reduced Stein space. 
Consider $\Om = \{\rho < 0\} \Subset X$ a strongly pseudoconvex domain and $\om_\Om$ a hermitian metric defined near the closure of $\Om$. 
Fix $0 \leq f \in L^p(\Om,\om_\Om^n)$. 
There exists a solution $\vph \in \PSH(\Om) \cap \CC(\overline{\Om})$ to the following complex Monge--Amp\`ere equation
\[
    (\ddc \vph)^n = f \om_\Om^n 
    \quad\text{with}\quad
    \vph_{|\pl\Om} = 0.
\]
Moreover, there is a constant $C_{\Om,p} = C(n, p, \Om, \om^n) > 0$ such that
\[
    \|\vph\|_{L^\infty(\Om)} \leq C_{\Om,p} \|f\|_{L^p(\Om,\om_\Om^n)}^{1/n}.
\]
\end{thm}

We next review a classical integrability fact:  

\begin{lem}\label{lem:Lp_pullback_met}
Let $X$ be a reduced complex analytic space of pure dimension $n \geq 1$ and $\pi: \hX \to X$ be a log resolution of singularities. 
Let $D_i$'s be irreducible components of the exceptional divisor of $\pi$. 
For each $i$, fix $s_i \in H^0(\hX,\CO(D_i))$ a section cutting out $D_i$ and $h_i$ a hermitian metric on $\CO(D_i)$.
Let $\dd V$ be a smooth volume form defined on an open set $U \subset X$.
For any compact subset $K \subset U$, there exist positive constants $\af_i = \af_i(K) > 0$ and $C_K > 0$ such that on $K$,
\[
    \prod_i |s_i|_{h_i}^{2\af_i} \dd V \leq C_K \pi^\ast \om^n.
\]
In particular, $f = \frac{\dd V}{\pi^\ast\om^n} \in L^l_{\loc}(U, \dd V)$ for any $l \in (1, l_0)$ where $l_0 = 1 + \frac{1}{\max_i \af_i}$. 
\end{lem}
\begin{proof}
The first assertion is a standard fact in several complex variables (see e.g. \cite[Proof of Thm.~1.1]{Chu_McCleerey_2021}).
We now verify the integrability of $f^l$. 
Let $E$ be the exceptional divisor of $\pi$. 
Fix a compact subset $K \subset \hX$. 
By the estimate between $\dd V$ and $\pi^\ast \om^n$, we have 
\[
    \int_{K \setminus D} \lt(\frac{\dd V}{\pi^\ast \om^n}\rt)^l \pi^\ast \om^n
    = \int_{K \setminus D} \lt(\frac{\dd V}{\pi^\ast \om^n}\rt)^{l-1} \dd V
    \leq \int_{K \setminus D} \lt(\frac{C_K^{l-1}}{\prod_i |s_i|^{2\af_i(l-1)}}\rt) \dd V.
\]
The right hand side is finite if $\af_i (l-1) < 1$ for all $i$; namely, $l < 1 + \frac{1}{\max_i(\af_i)} =: l_0 $.
\end{proof}

With the previous preparations, we now prove the resolution of generic choices of strongly pseudoconvex domains are K-domains.

\begin{lem}\label{lem:resoln_pscvx_Kdomain}
Let $X$ be an $n$-dimensional locally irreducible reduced Stein space, and let $\Om := \{\rho < 0\} \Subset X$ be a strongly pseudoconvex domain. 
Let $\pi: \hX \to X$ be a log resolution of singularities. 
Then $\pi^{-1}(\Om_r)$ is a K-domain for almost every $r \in (\min_\Om\rho, 0)$ where $\Om_r := \set{x \in \Om}{\rho(x) < r}$.
\end{lem}

\begin{proof}
Set $\wOm_r = \pi^{-1}(\Om_r)$.
We first check $\pl\wOm_r$ is smooth for almost every $r \in (\min_\Om\rho, 0)$.
Note that $\pi^\ast \rho$ is a smooth function defined near the closure of $\wOm$. 
By Sard's theorem, the critical value of $\pi^\ast \rho$ is a measure zero set in $(\min_\Om \rho, 0)$. 
Taking $r > 0$ sufficiently small but not a critical value, we obtain that $\pl \wOm_r$ is smooth on $\hX$.

\smallskip
We now verify the existence of continuous solutions and uniform estimates of solutions to the complex Monge--Amp\`ere equation with $L^p$-densities. 
Let $\om$ be a hermitian metric on $\Om$, and $\dd V$ be a smooth volume form on $\wOm$.
Fix a function $f \in L^p(\wOm_r, \dd V)$ for some $p>1$. 
From Lemma~\ref{lem:Lp_pullback_met}, 
there is a constant $l_0 > 1$ such that $g = \frac{\dd V}{\pi^\ast\om^n}$ belongs to $L^{l-1}(\wOm_r, \dd V)$ for all $l \in (1, l_0)$. 
Fix $l \in (1,l_0)$. 
Take $m = \frac{pl}{p+l-1} > 1$, $p' = p/m = \frac{l+p-1}{l} > 1$ and $q' = \frac{l-1}{m-1} = \frac{l+p-1}{p-1}>1$.
Note that $1 = 1/p' + 1/q'$.
By H\"older inequality, we get 
\begin{equation}\label{eq:pullback_pcvx_K-dom}
    \lt(\int_{\wOm_r} f^m g^m \pi^\ast \om^n\rt)^{\frac{1}{m}} 
    = \lt(\int_{\wOm_r} f^m g^{m-1} \dd V \rt)^{\frac{1}{m}}
    \leq \lt(\int_{\wOm_r} f^{mp'} \dd V\rt)^{\frac{1}{mp'}} \lt(\int_{\wOm_r} g^{(m-1)q'} \dd V\rt)^{\frac{1}{mq'}}.
\end{equation}
Let $D$ be the exceptional divisor of $\pi$. 
The map $\pi: \wOm_r \setminus D \to \Om_r^\reg$ is an isomorphism. 
We still denote by $f$ and $g$ for the trivial extension of $(\pi_{|\wOm_r \setminus D})_\ast f_{|\wOm_r \setminus D}$ and $(\pi_{|\wOm_r \setminus D})_\ast g_{|\wOm_r \setminus D}$ on $\Om_r$, respectively. 
The estimate \eqref{eq:pullback_pcvx_K-dom} ensures that $fg \in L^m(\Om_r, \om^n)$. 
Hence, by Theorem~\ref{thm:Kolodziej}, there exists a solution $\vph' \in \PSH(\Om_r) \cap \CC(\overline{\Om_r})$ to 
\[
    (\ddc \vph')^n = fg \om^n
    \quad\text{with}\quad
    \vph'_{|\pl\Om_r} = 0
\]
and there is a uniform constant $C_{\Om_r,m} >0$ such that 
\[
    \|\vph'\|_{L^\infty(\Om_r)} \leq C_{\Om_r,m} \|fg\|_{L^m(\Om_r,\om^n)}^{1/n} 
    \leq C_{\Om_r,m} C_1^{1/n} \|f\|_{L^p(\wOm_r, \dd V)}^{1/n},
\]
where $C_1 > 0$ is a constant so that $\lt(\int_{\wOm_r} g^{(m-1)q'} \dd V \rt)^{1/mq'} \leq C_1$.
Pulling back everything to $\wOm_r$, we obtain a solution $\vph = \pi^\ast\vph' \in \PSH(\wOm_r) \cap \CC(\overline{\wOm_r})$ solving
\[
    (\ddc \vph)^n = f \dd V
    \quad\text{with}\quad
    \vph_{|\pl\wOm_r} = 0
\]
and $\|\vph\|_{L^\infty(\wOm_r)} \leq C_p \|f\|_{L^p(\wOm_r, \dd V)}^{1/n}$ for some uniform constant $C_p > 0$.
\end{proof}

Combining Theorem~\ref{thm:loc_Sobolev} and Lemma~\ref{lem:resoln_pscvx_Kdomain}, one can derive the following local version of uniform Sobolev inequality: 

\begin{prop}\label{prop:unif_loc_Sobolev_resol}
Suppose that $\{\rho<0\} =: \Om \subset X$ is a strongly pseudoconvex domain.
Let $\pi: \hX \to X$ be a log resolution of singularities and $\dd V$ be a volume form defined near the closure of $\wOm := \pi^{-1}(\Om)$. 
Fix $p>1$ and $A>0$. For almost every $r \in (\min_\Om\rho, 0)$, and for any $q \in \lt(1,\frac{n}{n-1}\rt)$, there exista $C_S > 0$ (depend on $\Om_r$) such that for all $\om \in \CK(\wOm_r, \dd V, p, A)$ and for every $u \in \CC^\infty_c(\wOm_r)$, 
\[
    \lt(\int_{\wOm_r} |u|^{2q} \om^n\rt)^{1/q}
    \leq C_S \int_{\wOm_r} |\dd u|_\om^2 \om^n.
\]
\end{prop}

With the proposition above, we are able to establish uniform Sobolev and Poincar\'e inequalities along a perturbation path of the pullback metric on the desingularization of hermitian varieties: 

\begin{thm}[cf. Theorem~\ref{bigthm:unif_Sobolev}]\label{thm:unif_Sobolev_Poincare}
Let $(X,\om)$ be a compact hermitian variety and let $\pi: \hX \to X$ be a log resolution of singularities. 
Fix a hermitian metric $\homg$ on $\hX$. 
Consider an $\vep$-family of hermitian metrics $\om_\vep := \pi^\ast \om + \vep \homg$ for $\vep \in (0,1]$. 
Then the following properties hold: 
\begin{enumerate}
    \item For every $q \in (1,\frac{n}{n-1})$, there is a constant $C_S > 0$ such that for any $\vep \in (0,1]$, 
    \[
        \forall f \in L^2_1(\hX,\om_\vep), 
        \quad 
        \lt(\int_\hX |f|^{2q} \om_\vep^n\rt)^{1/q} 
        \leq C_S \lt(\int_\hX |\dd f|_{\om_\vep}^2 \om_\vep^n + \int_\hX |f|^2 \om_\vep^n\rt).
    \]
    \item There exists a constant $C_P > 0$ such that for every $\vep \in (0,1]$, 
    \[
        \forall f \in L^2_1(\hX,\om_\vep) \ \,\text{and}\ \, \int_\hX f \om_\vep^n = 0, 
        \quad 
        \int_\hX |f|^2 \om_\vep^n
        \leq C_P \int_\hX |\dd f|_{\om_\vep}^2 \om_\vep^n.
    \]
\end{enumerate}
\end{thm}

\begin{proof}
{\bf Part 1: Uniform Sobolev inequality.}
Let $\pi: \hX \to X$ be a log resolution of singularities. 
Since $X$ is compact, it can be covered by finitely many domains $(U_j)_j$, which are balls centered at some points; namely, for each $j$, there is a local embedding $\iota_j: X \xhookrightarrow[\loc.]{} \BC^N$ such that $U_j = \iota_j(X) \cap \BB_r$. 
Up to shrinking the radius, one can assume that $U_j = \Om_{j,1} \cup \cdots \cup \Om_{j,M}$ where $(\Om_{j,k})_{j,k}$ are locally irreducible and by Lemma~\ref{lem:resoln_pscvx_Kdomain}, $\pi^{-1}(\Om_{j,k})$ is a K-domain in $\hX$ for each $j,k$. 
To simplify the notation, we denote by $(\Om_i)_i$ (resp. $(\wOm_i)_i$) the collection for theses $\Om_{j,k}$'s (resp. $\wOm_{j,k}$).

\smallskip
Let $(\chi_i)_i$ be a partition of unity subordinate to the open cover $(\Om_i)_i$ and denote by $\widetilde{\chi}_i = \pi^\ast \chi_i$. 
For each $i$, we have a K\"ahler metric $\om_{i,K}$ defined near $\overline{\Om_i}$ which is quasi-isometric to $\om$ on $\Om_i$. 
On the resolution $\hX$, there is a closed $(1,1)$-form $\ta$ such that $\homg_{i,K} := \pi^\ast \om_{i,K} + \ta$ is a K\"ahler metric defined near $\overline{\wOm_i}$ and $\homg_{i,K}$'s are quasi-isometric to $\homg$ on $\wOm_i$. 
Consider $\om_{i,\vep} = \pi^\ast \om_{i,K} + \vep \homg_{i,K}$ an $\vep$-family of K\"ahler metrics on $\wOm_i$ for $\vep \in (0,1]$.
By Lemma~\ref{lem:Lp_pullback_met}, there exist $p > 1$ and $A_1 > 0$ such that $\om_{i,\vep} \in \CK(\wOm_i, \homg^n, p, A_1)$ for all $i$. 
Also, by construction, $\om_\vep := \pi^\ast\om + \vep \homg$ is uniformly quasi-isometric to $\om_{i,\vep}$ on each $\wOm_i$; namely, 
there is a constant $A_2 > 0$ such that for all $\vep \in (0,1]$, for any $i$,
\[
    A_2^{-1} \om_\vep \leq \om_{i,\vep}
    \leq A_2 \om_\vep.
\]
Note that $|\dd \widetilde{\chi}_i|_{\om_{\vep}}^2 \leq |\dd \widetilde{\chi}_i|_{\pi^\ast \om}^2 = \pi^\ast|\dd\chi_i|_{\om}^2$ which is uniformly bounded. 
Hence, there exist $p' > \frac{q}{q-1}$ and $A_3 > 0$ so that $(\om_\vep)_{\vep \in (0,1]} \subset \CH((\wOm_i)_i, (\widetilde{\chi}_i)_i, \homg^n, p, A_1, A_2, p', A_3)$. 
Then, the desired Sobolev inequality follows directly from Theorem~\ref{thm:global_Sobolev}.

\smallskip
\noindent{\bf Part 2: Uniform Poincar\'e inequality.}
It is more or less well-known that uniform Poincar\'e inequalities can be derived from uniform Sobolev inequalities. 
We include a proof based on the strategy of \cite[Prop.~3.2]{Ruan_Zhang_2011} here. 
Suppose that a uniform constant $C_P$ does not exist.
Namely, there exists a sequence of $L^2_1(\hX,\om_{\vep_j})$-functions $(\vph_j)_j$ satisfying
\[
    \int_{\hX} \vph_j \om_{\vep_j}^n = 0,
    \quad \int_{\hX} |\vph_j|^2 \om_{\vep_j}^n = 1, 
    \quad\text{and}\quad 
    \int_{\hX} |\dd \vph_j|^2_{\om_{\vep_j}} \om_{\vep_j}^n \leq 1/j
\]
where $\vep_j \to 0$ as $j \to +\infty$.
Note that for any compact subset $K \subset \hX \setminus D$, the metric $\om_\vep$ converges smoothly to the metric $p^\ast \om$ on a neighborhood of $K$ in $\hX \setminus D$; hence, $\om_\vep$ is uniformly equivalent to $p^\ast \om$ on $K$ for all $\vep \in (0,1]$.
Now, we fix a connected compact subset $K \subset \hX \setminus D$.
The assumptions imply that $(\vph_j)_j$ is a bounded sequence in $L^2_1(K)$. 
By Banach--Alaoglu theorem, there exists a $L^2_1(K)$-function $\vph$ such that $\vph_j$ converges to $\vph$ weakly in $L^2_1(K)$.
Moreover, $\vph_j$ converges to $\vph$ strongly in $L^2(K)$ by Rellich theorem.

\smallskip
By lower semi-continuity of norm, we have
\[
    0 \leq \norm{\dd \vph}_{L^2(K)} \leq \lim_{j \ra \infty} \norm{\dd \vph_j}_{L^2(K)} = 0.
\]
Hence, $\vph$ is a locally constant function on $K$ and thus a constant by connectedness.
Using H\"older inequality, one can derive
\begin{equation}\label{eq:Poincare_1}
    \abs{\int_{\hX \setminus D} \vph_j \om_{\vep_j}^n}
    \leq \Vol_{\om_{\vep_j}}^{1/2}(\hX \setminus K) \norm{\vph_j}_{L^2(\hX \setminus K)}
    \leq \Vol_{\om_{\vep_j}}^{1/2}(\hX \setminus K).
\end{equation}
Note that
\begin{equation}\label{eq:Poincare_2}
    \abs{\int_K \vph_j \om_{\vep_j}^n} = \abs{\int_{\hX \setminus K} \vph_j \om_{\vep_j}^n},
\end{equation}
as $\int_{\hV} \vph_j \om_{\vep_j}^n = 0$.
Since $\vph$ is constant, applying (\ref{eq:Poincare_1}) and (\ref{eq:Poincare_2}), we obtain
\begin{equation}\label{eq:Poincare_4}
\begin{split}
    \Vol_{p^\ast \om}(K) |\vph| 
    &= \abs{\int_K \vph (p^\ast \om)^n}
    = \lim_{j \to \infty} \abs{\int_K \vph_j \om_{\vep_j}^n}
    = \lim_{j \to \infty} \abs{\int_{\hX \setminus K} \vph_j \om_{\vep_j}^n}\\
    &\leq \lim_{j \to \infty} \Vol_{\om_{\vep_j}}^{1/2}
    (\hX \setminus K)
    = \Vol_{p^\ast \om}^{1/2}(\hX \setminus K).
\end{split}
\end{equation}
On the other hand, from the uniform Sobolev inequality obtained in Part 1, for any $q \in (1,\frac{n}{n-1})$, there is a uniform constant $C_S > 0$ such that
\[
    \lt(\int_{\hX} |\vph_j|^{2q}\om_{\vep_j}^n\rt)^{1/q}
    \leq C_S \lt(\int_\hX |\dd \vph_j|^2_{\om_{\vep_j}} \om_{\vep_j}^n + \int_\hX |\vph_j|^2 \om_{\vep_j}^n \rt)
    \leq C_S \lt(1/j + 1 \rt)
    \leq 2C_S.
\]
By H\"older inequality, one can derive
\begin{equation}\label{eq:Poincare_3}
    \int_{\hX \setminus K} |\vph_j|^2 \om_{\vep_j}^n
    \leq \lt(\int_{\hX} |\vph_j|^{2q} \om_{\vep_j}^n\rt)^{1/q} \Vol_{\om_{\vep_j}}^{1/p}(\hX \setminus K)
    \leq 2C_S \Vol_{\om_{\vep_j}}^{1/p}(\hX \setminus K)
\end{equation}
where $p > 1$ is such that $1/p + 1/q = 1$.
Thus,
\begin{equation}\label{eq:Poincare_5}
\begin{split}
    1 &= \lim_{j \ra \infty} \int_{\hX} |\vph_j|^2 \om_{\vep_j}^n
    = \lim_{j \to \infty} \int_{\hX \setminus K} |\vph_j|^2 \om_{\vep_j}^n
    + \lim_{j \to \infty} \int_{K} |\vph_j|^2 \om_{\vep_j}^n\\
    &\leq 2C_S \Vol_{p^\ast \om}^{1/p}(\hX \setminus K) + \int_K |\vph|^2 \pi^\ast \om^n
    = 2C_S \Vol_{p^\ast \om}^{1/p}(\hX \setminus K) + \Vol_{p^\ast \om}(K)|\vph|^2\\
    &\leq 2C_S \Vol_{p^\ast \om}^{1/p}(\hX \setminus K) + \frac{\Vol_{p^\ast \om}(\hX \setminus K)}{\Vol_{p^\ast \om}(K)}.
\end{split}
\end{equation}
Here the first inequality comes from (\ref{eq:Poincare_3}), and the second inequality is referred to (\ref{eq:Poincare_4}).
Since $\hX \setminus D$ is connected by the irreducibility of $X$, one can choose a sequence of connected compact subset $K_j \subset \hX \setminus D$ such that $K_j \subset K_{j+1}$ for all $j$ and $\Vol_{p^\ast \om}(\hX \setminus K_j) \xrightarrow[j \ra + \infty]{} 0$.
Then the right-hand side of (\ref{eq:Poincare_5}) tends to zero, and this yields a contradiction.
\end{proof}

\subsection{Singular Gauduchon metrics}\label{subsec:SGM}
In this section, we prove Theorem~\ref{bigthm:approx_and_conv}. 
Following our previous work \cite[Sec.~2.1]{Pan_2022}, we have a uniform estimate on Gauduchon factors: 

\begin{thm}\label{thm:Gauduchon_estimate}
Let $(X,\om)$ be an $n$-dimensional compact hermitian manifold, and $\rho$ be a Gauduchon factor with respect to $\om$ and normalized by $\inf_X \rho = 1$. 
Then there exists a constant $C_G > 0$ such that
\[
    \sup_X \rho \leq C_G.
\]
Moreover, the constant $C_G$ depends only on $n, B, V, q, C_S, C_P$ where
\begin{itemize}
    \item $V$ is the volume of $X$ with respect to $\om$; 
    \item $B>0$ is a constant such that
    $
        -B \om^n \leq \ddc \om^{n-1} \leq B \om^n;
    $
    \item $q > 1$ and $C_S > 0$ are Sobolev exponent and Sobolev constant respectively, such that
    \[
        \forall f \in L^2_1(X,\om),\quad
        \lt(\int_X |f|^{2q} \om^n\rt)^{1/q} 
        \leq C_S \lt(\int_X |\dd f|_\om^2 \om^n + \int_X |f|^2 \om^n\rt);
    \]
    \item $C_P > 0$ is a Poincar\'e constant so that
    \[
        \forall f \in L^2_1(X,\om)
        \ \, \text{with} \ \, \int_X f \om^n = 0,
        \quad \int_X |f|_\om^2 \om^n \leq C_P \int_X |\dd f|_\om^2 \om^n.
    \]
\end{itemize}
\end{thm}

Theorem~\ref{thm:Gauduchon_estimate} can be obtained by Moser's iteration technique, as detailed in \cite[Sec.~2.1]{Pan_2022}. 
While in {\it loc. cit.}, we used the standard Sobolev inequality (with $q = \frac{n}{n-1}$), the argument remains unchanged in the current setting. 

\smallskip
Let $(X,\om)$ be an $n$-dimensional compact hermitian variety and $\pi: \hX \to X$ be a log resolution of singularities. 
There exists a closed $(1,1)$-form $\ta$ coming from the negative combination of the curvature along $E_i$ such that $\homg = \pi^\ast \om + \ta$ is a hermitian metric on $\hX$ (see e.g. \cite[Prop.~3.2]{Fino_Tomassini_2009} for the standard argument). 
Consider an $\vep$-family of hermitian metrics $\om_\vep := \pi^\ast \om + \vep \homg$ for $\vep \in (0,1]$. 
By Gauduchon's theorem~\cite{Gauduchon_1977}, for each $\vep \in (0,1]$, there exists a smooth function $\rho_\vep$ on $\hX$ such that $\ddc (\rho_\vep \om_\vep^{n-1}) = 0$. 
In particular, $\rho_\vep$ is unique up to scaling; we shall normalized by $\inf_\hX \rho_\vep = 1$. 

\begin{thm}[cf. Theorem~\ref{bigthm:approx_and_conv}]\label{thm:existence_bdd_Gauduchon}
Under the above setting, the following hold: 
\begin{enumerate}
    \item There is a uniform constant $C_G > 0$ independent of $\vep \in (0,1]$ such that $\sup_\hX \rho_\vep \leq C_G$;
    \item As $\vep \to 0$, one can extract a subsequence $\rho_{\vep_k}$ converging in $\CC^\infty_{\loc}(\hX \setminus D)$ towards a function $\rho \in \CC^\infty_\loc(\hX \setminus D) \cap L^\infty(\hX \setminus D)$ which is also bounded away from zero.
\end{enumerate}
In particular, since $\pi: \hX \setminus D \to X^\reg$ is an isomorphism, $\rho$ can descend to $X^\reg$, and $\om_\RG := \rho^{\frac{1}{n-1}} \om$ is a bounded Gauduchon metric on $X^\reg$. 
Furthermore, $\rho$ is the unique bounded solution to $\ddc(\rho \om^{n-1}) = 0$ on $X^\reg$ up to scaling. 
\end{thm}

\begin{proof}[Proof of Theorem~\ref{thm:existence_bdd_Gauduchon}]
The proof follows the approach given in the author's PhD thesis \cite[Sec.~3.2]{Pan_2023_these}. 

\smallskip
\noindent{\bf Step 1: Uniform $L^\infty$-estimate.}
To get a uniform $L^\infty$-estimate on $\rho_\vep$, it suffices to show that the constants $V$, $B$, $q$, $C_S$, $C_P$ in Theorem~\ref{thm:Gauduchon_estimate} can be made independent of $\vep \in (0,1]$. 
Obviously, the volume is bounded below by $\Vol_\om(X)$ and above by $\Vol_{\om_1}(\hX)$.
By Theorem~\ref{thm:unif_Sobolev_Poincare}, we have a uniform control of $q, C_S$ and $C_P$ independent of $\vep \in (0,1]$. 
It suffices to check the existence of a constant $B>0$ such that $-B \om_\vep^n \leq \ddc \om_\vep^{n-1} \leq B \om_\vep^n$ for all $\vep \in (0,1]$.

\smallskip
On $(X,\om)$, there is a constant $B' > 0$ such that 
\begin{equation}\label{eq:constant_B'}
    -B' \om^2 \leq \ddc\om \leq B' \om^2
    \quad \text{and} \quad
    -B' \om^3 \leq \dd\om \w \dc\om \leq B' \om^3.
\end{equation}
By the binomial expansions, we can express $\ddc \om_\vep^{n-1}$ and $\om_\vep^n$ as
\[
    \ddc \om_\vep^{n-1} 
    = \sum_{j=0}^{n-1} \vep^j \binom{n-1}{j} \underbrace{\ddc\lt(\pi^\ast \om^{n-1-j} \w \homg^j\rt)}_{=: \af_j},
    \quad\text{and}\quad
    \om_\vep^n 
    = \sum_{j=0}^n \vep^j \binom{n}{j} \underbrace{\pi^\ast \om^{n-j} \w \homg^j}_{=: \bt_j \geq 0}.
\]
Hence, to find a constant $B > 0$ satisfying $- B \om_\vep^n \leq \ddc \om_\vep^{n-1} \leq B \om_\vep^n$, it suffices to prove that for all $j \in \{0, \cdots, n-1\}$, one has $-\bt_j \lesssim \af_j \lesssim \bt_j$.
Note that
\[
    \af_j = \ddc\lt(\pi^\ast \om^{n-1-j} \w \homg^j\rt)
    = \underbrace{\ddc \pi^\ast \om^{n-1-j} \w \homg^j}_{=: \af_{j,1}} 
    + \underbrace{\pi^\ast \om^{n-1-j} \w \ddc \homg^j}_{=: \af_{j,2}} 
    + 2 \underbrace{\dd \pi^\ast \om^{n-1-j} \w \dc \homg^j}_{=: \af_{j,3}}.
\]
Since $\dd \homg = \dd \pi^\ast \om$, $\dc \homg = \dc \pi^\ast \om$ and $\ddc \homg = \ddc \pi^\ast \om$, we have
\begin{align*}
    \af_{j,1} &= (n-1-j) \lt(\ddc \pi^\ast \om \w \pi^\ast \om^{n-2-j}\rt) \w \homg^j\\
    &\quad+ (n-1-j)(n-2-j) \lt(\dd \pi^\ast \om \w \dc \pi^\ast \om \w \pi^\ast \om^{n-3-j}\rt) \w \homg^j,
\end{align*}
\[
    \af_{j,2}  = \underbrace{j \lt(\ddc\pi^\ast \om \w \pi^\ast \om^{n-1-j}\rt) \w \homg^{j-1}}_{=: \af_{j,2}'}
    + \underbrace{j(j-1) \lt(\dd\pi^\ast \om \w \dc\pi^\ast \om \w \pi^\ast \om^{n-1-j}\rt) \w \homg^{j-2}}_{=: \af_{j,2}''},
\]
and 
\[
    \af_{j,3} = j (n-1-j) \lt(\dd\pi^\ast \om \w \dc\pi^\ast \om \w \pi^\ast \om^{n-2-j}\rt) \w \homg^{j-1}.
\]
Using (\ref{eq:constant_B'}), one can deduce that there is a constant $C = C(n, B')$ such that for all $j$,  
\[
    - C \bt_j \leq \af_{j,1} \leq C \bt_{j},\quad
    -  C \bt_{j-2} \leq \af_{j,2}', \af_{j,3} \leq C \bt_{j-2},\quad
    - C \bt_{j-1} \leq \af_{j,2}'' \leq C \bt_{j-1}
\]
Note that the power of $\vep$ in front of $\af_{j,1}$ (resp. $\af_{j,2}'$, resp. $\af_{j,2}'', \af_{j,3}$) is less than or equal to the power of $\vep$ in front of $\bt_j$ (resp. $\bt_{j-2}$, resp. $\bt_{j-1}$).
Thus, there exists a uniform constant $B > 0$ which depends only on $n$, $B'$ such that for all $\vep \in (0,1]$, 
\begin{equation}\label{eq:desing_const_B}
    -B \om_\vep^n 
    \leq \ddc \om_\vep^{n-1} 
    \leq B \om_\vep^n.
\end{equation}
Therefore, by Theorem~\ref{thm:Gauduchon_estimate}, there is a uniform constant $C_G > 0$ such that
\begin{equation}\label{eq:unif_bdd_Gauduchon_factor}
    \sup_\hX \rho_\vep \leq C_G
\end{equation}
for all $\vep \in (0,1]$.

\smallskip
\noindent{\bf Step 2: Convergence of $\rho_\vep$.}
Then we follow \cite[Thm.~3.1]{Pan_2022} to construct a smooth function $\rho$ on $\hX \setminus E$.
This $\rho$ is the limit of $(\rho_{\vep_j})_{j \in \BN}$ for a sequence $\vep_j \ra 0$ when $j \ra +\infty$.

\smallskip
We set $P_\vep = \Dt_{\om_\vep}^\ast$ and fix $U_1 \Subset U_2 \Subset \hX \setminus D$ which are connected open subsets.
On $U_2$, the metric $\pi^\ast \om$ is quasi-isometric to the metric $\om_\vep$, and the volume form $\pi^\ast \om^n$ is comparable with $\om_\vep^n$.
In other words, we have a uniform constant $C_{U_2} > 0$ such that for all $\vep \in (0,1]$,
\begin{equation}\label{eq:loc_met_bound_Ch3}
    C_{U_2}^{-1} \iprod{\cdot}{\cdot}_{\pi^\ast \om} 
    \leq \iprod{\cdot}{\cdot}_{\om_\vep} 
    \leq C_{U_2} \iprod{\cdot}{\cdot}_{\pi^\ast \om},
    \text{ and }
    C_{U_2}^{-1} \pi^\ast \om^n 
    \leq \om_\vep^n 
    \leq C_{U_2} \pi^\ast \om^n
\end{equation}
on $U_2$.

\smallskip
By G\r{a}rding inequality, we have
\[
    \norm{u}_{L^2_2(U_1, \pi^\ast \om)} 
    \leq C_{U_1, U_2} \lt( \norm{P_\vep u}_{{L^2(U_2, \pi^\ast \om)}} + \norm{u}_{{L^2(U_2, \pi^\ast \om)}} \rt)
\]
for all $u \in \CC^\infty_c(U_2)$.
The constant $C_{U_1,U_2}$ can be chosen independent of $\vep$ because the coefficients of $P_{\vep}$ move smoothly in $\vep$ on $\overline{U_2}$.
Choose a cut-off function $\chi$ such that $\supp(\chi) \subset U_2$ and $\chi \equiv 1$ on $U_1$.
From G\r{a}rding inequality, it follows that
\begin{equation}\label{eq:rho_Garding_1_Ch3}
    \norm{\rho_\vep}_{{L^2_2(U_1, \pi^\ast \om)}} 
    \leq C_{U_1,U_2} \lt( \norm{P_\vep (\chi \rho_\vep)}_{{L^2(U_2, \pi^\ast \om)}} 
    + \norm{\chi \rho_\vep}_{{L^2(U_2, \pi^\ast \om)}} \rt).
\end{equation}
In (\ref{eq:rho_Garding_1_Ch3}), the second term $\norm{\chi \rho_\vep}_{{L^2(U_2, \pi^\ast \om)}}$ is uniformly bounded by \eqref{eq:unif_bdd_Gauduchon_factor}.
Hence, we only need to estimate $\norm{P_\vep (\chi\rho_\vep)}_{{L^2(U_2, p^\ast \om_V)}}$.
Note that
\begin{equation}\label{eq:Pt_cut_eq_Ch3}
\begin{split}
    P_\vep(\chi \rho_\vep) 
    &= 2 \iprod{\dd\rho_\vep}{\dd\chi}_{\om_\vep} + \rho_\vep P_\vep(\chi) - \rho_\vep \chi \frac{\ddc\om_\vep^n}{\om_\vep^n}.
\end{split}
\end{equation}
The last two terms on the RHS, $\rho_\vep P_\vep(\chi)$ and $\rho_\vep \chi \frac{\ddc\om_\vep^{n-1}}{\om_\vep^n}$, are uniformly bounded, so we only need to control the $L^2$-norm of the first term $\iprod{\dd\rho_\vep}{\dd\chi}_{\om_\vep}$:
{\small
\begin{equation*}
\begin{split}
    \int_{U_2} \abs{\iprod{\dd\rho_\vep}{\dd\chi}_{\om_\vep}}^2 \pi^\ast \om^n
    &\leq C_{U_2}^2 \lt(\sup_{U_2} \abs{\dd\chi}^2_{\pi^\ast \om}\rt) \int_{U_2} \abs{\dd\rho_\vep}_{\om_\vep}^2 \om_\vep^n 
    \leq C_{U_2}^2 \lt(\sup_{U_2} \abs{\dd\chi}^2_{\pi^\ast \om}\rt) \int_\hX \abs{\dd\rho_\vep}^2_{\om_\vep} \om_\vep^n\\
    &\leq C_{U_2}^2 \lt(\sup_{U_2} \abs{\dd\chi}^2_{\pi^\ast \om}\rt) \frac{nB}{2} \int_\hX \rho_\vep^2 \om_\vep^n.
\end{split}
\end{equation*}
}%
Here, the first line is by Cauchy--Schwarz inequality and (\ref{eq:loc_met_bound_Ch3}).
The third line follows from \cite[Lem.~2.4]{Pan_2022} by taking $F'(x) = n$ and $G(x) = \frac{n x^2}{2}$.
By \eqref{eq:unif_bdd_Gauduchon_factor}, one can then find a uniform bound of $\norm{P_\vep(\chi\rho_\vep)}_{L^2(U_2, \pi^\ast \om)}$. 
Hence, $\norm{\rho_\vep}_{L^2_2(U_1,\pi^\ast \om)}$ is uniformly bounded by some uniform constant $C(U_1,U_2)$.

\smallskip
For higher order estimates, we apply higher order G\r{a}rding inequalities on the fixed domains $U_1 \Subset U_2 \Subset \hX \setminus E$:
\[
    \norm{u}_{L^2_{s+2}(U_1,\pi^\ast \om)} 
    \leq C_{s, U_1, U_2}\lt( \norm{P_\vep u}_{L^2_s(U_2, \pi^\ast \om)} + \norm{u}_{L^2(U_2, \pi^\ast \om)} \rt)
\]
for all $u \in \CC^\infty_c(U_2)$.
Let $\CU = (U_i)_{i\in\BN}$ be a relatively compact exhaustion of $\hX \setminus D$.
Differentiating (\ref{eq:Pt_cut_eq_Ch3}) on both sides and using a bootstrapping argument, we obtain $\norm{\rho_\vep}_{L^2_s(U_1,\pi^\ast \om)} < C(s,\CU)$ where $C(s,\CU)$ does not depend on $\vep \in (0,1]$.
By Rellich's theorem, there exists a subsequence $(\rho_{\vep_j})_{j \in \BN}$ such that $\rho_{\vep_j}$ converges to $\rho$ in $\CC^l(\overline{U_1})$ for all $l \in \BN$ when $\vep_j \ra 0$.
Therefore $\ddc(\rho \pi^\ast \om^{n-1}) = \lim_{j \ra + \infty} \ddc (\rho_{\vep_j} \om_{\vep_j}^{n-1}) = 0$ on $U_1$.
Using a diagonal argument, we infer that there is a smooth function $\rho$ on $\hX \setminus D$. 
Recall that $\pi_{|\hX \setminus D}: \hX \setminus D \to X^\reg$ s an isomorphism; hence, $\rho$ can descend to $X^\reg$ as a smooth function bounded between $1$ and $C_G$, and it satisfies $\ddc(\rho \om^{n-1}) = 0$ on $X^\reg$. 
\end{proof}

\section{Slope stability and openness property}\label{sec:slope_openness}
We introduce the notion of slope stability with respect to a bounded Gauduchon metric and then show the openness property for stable sheaves on a resolution of singularities. 

\subsection{Definition of stability}\label{sec:defn_stability}
Fix $X$ an $n$-dimensional compact complex variety, and we now further assume that $X$ is normal. 

\smallskip
Let $\om$ be a smooth hermitian metric on $X$.
By Theorem~\ref{thm:existence_bdd_Gauduchon}, there exists a positive smooth function $\rho$ bounded away from zero and $+\infty$ on $X^\reg$ such that $\ddc (\rho \om^{n-1}) = 0$ on $X^\reg$. 
We denote by $\om_\RG = \rho^{\frac{1}{n-1}} \om$ a bounded Gauduchon metric. 
By \cite[Thm.~B]{Pan_2022}, $T_\RG := \om_\RG^{n-1}$ can be extended trivially to a pluriclosed current on $X$, i.e. $\ddc T_\RG = 0$ in the sense of currents. 

\smallskip
Let $\CF$ be a torsion-free coherent sheaf of rank $r$ on $X$. 
Take $f: Y \to X$ a log resolution such that $f_{|f^{-1}(X^\reg)}$ is an isomorphism. 
Denote by $f^\sharp \CF := f^\ast \CF / \Tor(f^\ast \CF)$ which is torsion-free.
Note that $f^\ast T_\RG$ is a $\ddc$-closed current and it represents a class in the $(n-1,n-1)$ Aeppli cohomology group $H^{n-1,n-1}_{\mathrm{A}}(Y,\BR)$. 
The degree of $\CF$ is defined as 
\[
    \deg_{\om_\RG}(\CF) := \int_Y \BCC(\det(f^\sharp \CF)) \w f^\ast T_\RG
\]
where $\BCC(\det(f^\sharp \CF)) \in H^{1,1}_{\mathrm{BC}}(Y,\BR)$ is the first Bott--Chern class of $\det(f^\sharp \CF) := (\bigwedge^r f^\sharp \CF)^{\vee\vee}$. The following observation shows the well-definedness of the degree (cf. \cite[page~773]{Claudon_Graf_Guenancia_2022}). 
Suppose that there is another resolution $g: W \to X$. 
One can find a resolution $h: W' \to X$ that factors as $W' \xrightarrow{g'} Y \xrightarrow{f} X$ and $W' \xrightarrow{f'} W \xrightarrow{g} X$. 
It therefore suffices to verify the case where the resolution factors as $Y' \xrightarrow{g} Y \xrightarrow{f} X$. 
Note that $\det(h^\sharp \CF)$ and $g^\ast \det(f^\sharp \CF)$ are isomorphic away from the exceptional divisor of $g$ which is denoted by $G = \sum_i G_i$. 
Thus, $\BCC(\det(h^\sharp \CF)) - g^\ast \BCC(\det(f^\sharp \CF))$ is a linear combination of $G_i$.
Since $\int_{G_i} h^\ast T_\RG = 0$ (see the proof of Lemma~\ref{lem:pullback_stable}), we get the following
\[
    \int_{Y'} \BCC(\det (h^\sharp \CF)) \w h^\ast T_\RG
    = \int_{Y'} g^\ast \BCC(\det (f^\sharp \CF)) \w h^\ast T_\RG
    = \int_{Y} \BCC(\det(f^\sharp \CF)) \w f^\ast T_\RG.
\]

\begin{defn}
For a torsion-free coherent sheaf $\CF$ on $X$, the corresponding notion of $\mu_{\om_\RG}$-slope is defined by 
\[
    \mu_{\om_\RG}(\CF) := \frac{\deg_{\om_\RG}(\CF)}{\rk(\CF)}.
\]
A torsion-free coherent sheaf $\CF$ is {\it $\mu_{\om_\RG}$-stable} (resp. {\it $\mu_{\om_\RG}$-semi-stable}) if for every proper coherent subsheaf $\CG$ of $\CF$, 
\[
    \mu_{\om_\RG}(\CG) < \mu_{\om_\RG}(\CF) \quad \text{(resp. $\mu_{\om_\RG}(\CG) \leq \mu_{\om_\RG}(\CF)$)}.
\]
\end{defn}

\subsection{Resolution and openness of stability}
For a torsion-free coherent sheaf $\CF$ of rank $r$, define
\[
    Z(X,\CF) := X^\sing \cup \Sing(\CF) 
\]
where $\Sing(\CF) = \{x \in X \;|\; \text{$\CF$ is not locally free near $x$}\}$. 
Note that $Z(X,\CF)$ has codimension at least $2$. 

\smallskip
In the sequel, we shall always consider the following setting: 

\begin{setupnota}\label{setup:reflexive_sheaf_bdd_Gauduchon}
Let $X$ be an $n$-dimensional compact normal variety, and let $\CE$ be a coherent reflexive sheaf of rank $r$ on $X$. 
Define $Z := Z(X,\CE)$, and $X^\circ := X \setminus Z$.
By Hironaka and Rossi \cite{Hironaka_1964, Rossi_1968}, we can fix a log resolution $\pi: \hX \to X$ of $(X,Z)$ such that $E := \pi^\sharp\CE$ is locally free.
Set $D := \Exc(\pi)$.
Note that $(\pi_\ast E)^{\vee \vee} \simeq \CE$.  

\smallskip
Consider a smooth hermitian metric $\om$ on $X$. 
By Theorem~\ref{bigthm:approx_and_conv}, there is a bounded Gauduchon metric $\om_\RG = \rho^{\frac{1}{n-1}} \om$ on $X^\reg$ so that $\int_X \om_\RG^n = 1$. 

\smallskip
We shall denote by $\Tr_{\End}$ for the trace acting on $\CA^{p,q}(\hX, \End E)$ the $\End(E)$-valued smooth $(p,q)$-forms, and $\tr_{\eta}$ for the contraction on (bundle-valued) forms with respect to a hermitian metric $\eta$.
\end{setupnota}

Recall that, as in Section~\ref{subsec:SGM}, there is a smooth closed $(1,1)$-form $\ta$ on $\hX$ such that $\homg := \pi^\ast \om + \ta$ is a hermitian metric on $\hX$ and for $\vep \in (0,1]$, define $\om_\vep := \pi^\ast \om + \vep \homg$. 
Gauduchon's theorem yields a unique positive function $\rho_\vep \in \CC^\infty(\hX)$ satisfying $\ddc (\rho_\vep \om_\vep^{n-1}) = 0$ and $\int_\hX \om_{\vep,\RG}^n = 1$ where $\om_{\vep, \RG} := \rho_\vep^{\frac{1}{n-1}} \om_\vep$. 
From Theorem~\ref{thm:existence_bdd_Gauduchon}, we have shown that $(\rho_\vep)_{\vep \in (0,1]}$ are uniformly bounded from above and away from zero, and $\rho_\vep$ converges in $\CC^\infty_{\loc}(\hX \setminus D)$ towards a function $\rho \in \CC^\infty(X^\reg) \cap L^\infty(X^\reg)$ so that $\ddc (\rho \cdot \pi^\ast \om^{n-1}) = 0$ on $\hX \setminus D$.
As $\pi_{|\hX \setminus D}: \hX \setminus D \to X^\reg$ is an isomorphism, $\rho$ can descend to $X^\reg$ and $\om_{\RG} = \rho^{\frac{1}{n-1}} \om$ due to the uniqueness. 

\smallskip
We recall the following theorem of Cao--Graf--Naumann--P\u{a}un--Peternell--Wu \cite{Paun_et_al_2023_HE_sing}: 

\begin{thm}[{\cite[Thm.~1.1]{Paun_et_al_2023_HE_sing}}]\label{thm:Mihai_et_al_thm1}
There exists a hermitian metric $h_E$ on $E$ and a constant $C > 0$ such that
\[
    \ii \Ta(E,h_E) \geq -C \pi^\ast \om \otimes \Id_E.
\]
\end{thm}

Since the proof of \cite[Thm.~1.1]{Paun_et_al_2023_HE_sing} is purely local, it works also in the hermitian setting. 
In Theorem~\ref{thm:existence_bdd_Gauduchon}, we have showed that $\om_\vep$ (resp. $\om$) and $\om_{\vep,\RG}$ (resp. $\om_\RG$) are uniformly quasi-isometric for all $\vep \in (0,1]$. 
Combining Theorem~\ref{thm:Mihai_et_al_thm1} and Theorem~\ref{thm:existence_bdd_Gauduchon}, one obtains the following: 

\begin{cor}\label{cor:Mihai_et_al_thm1}
There exists a hermitian metric $h_E$ on $E$ and a constant $C > 0$ such that, for all $\vep \in (0,1]$, 
\[
    \ii \Ta(E,h_E) \geq -C \om_{\vep,\RG} \otimes \Id_E
    \quad\text{and}\quad
    \ii \Ta(E,h_E) \geq -C \pi^\ast (\om_\RG) \otimes \Id_E.
\]
\end{cor}

\subsubsection{Degree and curvature}
Let $h_E$ be a hermitian metric on $E$. 
From \cite[Prop.~2.3]{Bruasse_2001}, for a subsheaf $\CG$ of $E$, the degree of $\CG$ can be obtained by 
\begin{equation}\label{eq_deg_subsheaf}
    \deg_{\om_{\vep, \RG}}(\CG) 
    = \frac{\ii}{2\pi} \int_\hX \Tr_{\End}\lt(\Ta(\CG,{h_E}_{|\hX \setminus \Sing(\CG)})\rt) \w \om_{\vep,\RG}^{n-1}
\end{equation}
where $\Sing(\CG)$ is the non-locally free locus of $\CG$. 
Following the same argument, one can also express $\deg_{\pi^\ast \om_\RG}(\CG)$ as 
\[
    \deg_{\pi^\ast \om_\RG}(\CG) 
    = \frac{\ii}{2\pi} \int_{\hX} \Tr_{\End}(\Ta(\CG, {h_E}_{|\hX \setminus \Sing(\CG)})) \w \pi^\ast \om_\RG^{n-1}.
\]

\begin{lem}\label{lem:pullback_stable}
Under Setup~\ref{setup:reflexive_sheaf_bdd_Gauduchon}, $\CE$ is $\mu_{\om_\RG}$-stable if and only if $E$ is $\mu_{\pi^\ast \om_\RG}$-stable.
\end{lem}

\begin{proof}
Let $(D_i)_i$ be the irreducible components of the exceptional divisor of $\pi$. 
We first claim that for any $D_i$,
\begin{equation}\label{eq:zero_slope}
    \int_Y \BCC(\CO(D_i)) \w \pi^\ast \om_\RG^{n-1} = 0.
\end{equation}
Note that $\int_Y \BCC(\CO(D_i)) \w \pi^\ast \om_\RG^{n-1} = \lim_{\vep \to 0^+} \int_Y \BCC(\CO(D_i)) \w \om_{\vep,\RG}^{n-1} = \lim_{\vep \to 0^+} \int_{D_i} \om_{\vep,\RG}^{n-1}$ and 
\[
    \int_{D_i} \om_{\vep,\RG}^{n-1} 
    \leq C_G \int_{D_i} \om_\vep^{n-1} \xrightarrow[\vep \to 0^+]{}
    C_G \int_{D_i} \pi^\ast \om^{n-1} = 0.
\]
This confirms \eqref{eq:zero_slope}. 
Then, with \eqref{eq:zero_slope}, following the argument in \cite[Lem.~3.2]{Claudon_Graf_Guenancia_2022}, one can conclude. 
\end{proof}

Denote by $\af_\vep := [\om_{\vep,\RG}^{n-1}]_\RA \in H_{\RA}^{n-1,n-1}(\hX,\BR)$ the Aeppli cohomology class of $\om_{\vep,\RG}^{n-1}$. 
Recall that $\rho_\vep$ converges locally smoothly to $\rho$ on $\hX \setminus D$, and $(\rho_\vep)_{\vep \in (0,1]}$ is uniformly bounded by Theorem~\ref{bigthm:approx_and_conv}. 
Given any $\bt \in H_{\RBC}^{1,1}(\hX,\BR)$, for all representative $b \in \bt$, we have $\int_\hX b \w \om_{\vep,\RG}^{n-1} \xrightarrow[\vep \to 0]{} \int_\hX b \w \pi^\ast \om_\RG^{n-1}$. 
Hence, $\af_\vep$ converges weakly to $\af_0 := [\pi^\ast \om^{n-1}_\RG]_\RA$. 
By the compactness of $\hX$, $H_{\RA}^{n-1,n-1}(\hX,\BR)$ is finite dimensional. 
Thus, $\af_\vep$ converges strongly to $\af_0$ and 
\[
    \lim_{\vep \to 0} \deg_{\om_{\vep,\RG}}(E) = \deg_{\pi^\ast \om_\RG}(E) = \deg_{\om_\RG}(\CE)
    \quad\text{and}\quad
    \lim_{\vep \to 0} \mu_{\om_{\vep,\RG}}(E) = \mu_{\pi^\ast \om_\RG}(E) = \mu_{\om_\RG}(\CE).
\]

\subsubsection{Openness of stability}
In the K\"ahler setting, the openness of stability is a well-known result from the theory of sheaves (see e.g. \cite[Cor.~6.3]{Toma_2021} and \cite[Lem.~3.3 and Prop.~3.4]{Claudon_Graf_Guenancia_2022}). 
However, there is no such result in the hermitian context. 
We shall follow the argument in \cite[Sec.~2]{Bruasse_2001} and \cite[Thm.~1.2]{Chen_Zhang_2024} to obtain the following openness property: 

\begin{thm}\label{thm:openness_stability}
Under Setup~\ref{setup:reflexive_sheaf_bdd_Gauduchon}, if $\CE$ is $\mu_{\om_\RG}$-stable, then $E$ is $\mu_{\om_{\vep,\RG}}$-stable for all $\vep > 0$ sufficiently small. 
\end{thm}

\begin{proof}
Arguing by contradiction, we assume that there exists a sequence $\vep_j \to 0$ as $j \to +\infty$ and a sequence of proper subsheaves $\CF_j$ of $E$ such that 
\[
    \mu_{\om_{\vep_j,\RG}}(\CF_j) \geq \mu_{\om_{\vep_j,\RG}}(E)
\]
Since $\rho_\vep$ is uniformly bounded, there is a constant $C>0$ such that the following inequality holds 
\begin{equation}\label{eq:open:deg_lower_bdd}
    \deg_{\om_{\vep_j,\RG}}(\CF_j) 
    \geq \frac{\rk(\CF_j)}{\rk(E)} \deg_{\om_{\vep_j,\RG}}(E) \geq -C
\end{equation}
for all $j$. 
One may assume that $\rk(\CF_j) = r' \in \{1, \cdots, r-1\}$ for all $j$ up to extracting a subsequence. 

\smallskip
Denote by $F_j$ the locally free part of $\CF_j$.
By \eqref{eq_deg_subsheaf}, and the Gauss--Codazzi equation of Chern connections on vector bundles (see e.g. \cite[p.~22-23]{Kobayashi_book} and \cite[p.~88]{Lubke_Teleman_1995})
\begin{equation}\label{eq:open_GC_eqn}
\begin{split}
    \deg_{\om_{\vep_j,\RG}}(\CF_j) 
    &= \int_\hX \Tr_{\End}(\tr_{\om_{\vep_j,\RG}} \ii\Ta(F_j,{h_E}_{|F_j})) \om_{\vep_j,\RG}^n \\
    &= \int_\hX \lt( \Tr_{\End}(p_j \circ \ii\Ta(E,h_E) \circ p_j) - \Tr_{\End}(\ii \pl p_j \w \db p_j)\rt) \w \om_{\vep_j,\RG}^{n-1}
\end{split}
\end{equation}
where $p_j$ is the $h_E$-orthogonal projection $p_j: E \to F_j$ on $\hX \setminus \Sing(\CF_j)$.
Note that $|p_j|_{h_E} = \sqrt{r'}$; we thus have
\begin{equation}\label{eq:open_curv_L1_bdd}
    \int_\hX \Tr_{\End}(p_j \circ \ii\Ta(E,h_E) \circ p_j) \w \om_{\vep,\RG}^{n-1} 
    \leq r' \int_\hX |\Ta(E,h_E) \w \om_\vep^{n-1}|_{h_E}
    \leq C
\end{equation}
for some constant $C > 0$ independent of $\vep$. 
In the above inequality, we use $\om_{\vep,\RG}$ are uniformly quasi-isometric to $\om_\vep$ for all $\vep \in (0,1]$ from Theorem~\ref{bigthm:approx_and_conv}.

\smallskip
Combining \eqref{eq:open:deg_lower_bdd}, \eqref{eq:open_GC_eqn}, \eqref{eq:open_curv_L1_bdd}, we obtain 
\[
    \int_\hX |\db p_j|_{h_E, \om_{\vep_j}}^2 \om_{\vep_j}^n \leq C.
\]
Since $\pi^\ast \om \leq \om_\vep$, one can infer that $(p_j)_j$ are uniformly bounded in $L^2_1(\hX \setminus D, \End(E_{|\hX \setminus D}), \pi^\ast \om, h_E)$. 
The weak compactness, lower semi-continuity of $L^2_1$-norm along weak convergence sequences, and Rellich's theorem show that, after extracting a subsequence, $p_j$ converges weakly in $L^2_{1,\loc}$ and strongly in $L^2_{\loc}$ towards $p_\infty \in L^2_1(X^\circ, \End(E_{|\hX \setminus D}), \pi^\ast \om, h_E)$.
By the regularity of $L^2_1$-subbundle (see \cite{Uhlenbeck_Yau_1986, Popovici_2005}), $p_\infty$ induces a coherent subsheaf $\CF_\infty$ of $E_{|\hX \setminus D}$. 
Since $\pi_{|\hX \setminus D}: \hX \setminus D \to X^\circ$ is an isomorphism and $E_{|\hX \setminus D}$ and $\CE_{|X^\circ}$ are isomorphic, $\CF_\infty' := (\pi_{|\hX \setminus D})_\ast \CF_\infty$ is a coherent subsheaf of $\CE_{|X^\circ}$ on $X^\circ$. 

\smallskip
Following the idea of Uhlenbeck and Yau \cite[Sec.~7]{Uhlenbeck_Yau_1986} (see also \cite[proof of Prop.~3.3]{Chen_2022}), $\CF_\infty'$ can be extended as a coherent subsheaf on $X$ as $Z$ has codimension at least two. 
Since the argument is local, one can consider $\CF'_\infty$ as a subsheaf of $\CO_X^{\oplus N}$ over $U \setminus Z$.
Then $\CF_\infty'$ induces a meromorphic map $f: U \setminus Z \to \Gr(N, N-r')$.
Denote by $\Gm_f$ graph of $f$ in $(U \setminus Z) \times \Gr(N, N-r')$.
By Siu's theroem \cite[bottom of p.~441]{Siu_1975}, the closure of $\Gm_f$, $\overline{\Gm_f}$, is a subvariety in $U \times \Gr(N, N-r')$. 
Let $\pr_1$ and $\pr_2$ be the projections from $\overline{\Gm_f}$ to $U$ and $\Gr(N, N-r')$, respectively. 
Set $\mathfrak{F}$ the tautological bundle on $\Gr(N, N-r')$. 
Then $(\pr_1)_\ast \pr_2^\ast \mathfrak{F}$ provides a coherent extension of $\CF_\infty'$ to $U$ and its inverse image is an extension of $\CF_\infty$ on $\hX$.
Thus,
\begin{align*}
    \deg_{\pi^\ast \om_{\RG}} (\CF_\infty) 
    &= \int_\hX (\Tr_{\End} (p_\infty \circ \ii\Ta(E,h_E) \circ p_\infty) - \Tr_{\End}(\ii \pl p_\infty \w \db p_\infty))\w \pi^\ast\om_\RG^{n-1} \\
    &\geq \frac{\rk(\CF_\infty)}{\rk(E)} \deg_{\pi^\ast \om_\RG}(E).
\end{align*}
However, this contradicts the $\mu_{\pi^\ast \om_\RG}$- stability of $E$.
\end{proof}

\section{Singular Hermite--Einstein metrics on stable sheaves}\label{sec:HE_estimates}
In this section, we prove Theorem~\ref{bigthm:HE_metric}.

\subsection{Harnack estimates} 
We first establish a uniform Harnack-type estimate. 

\begin{lem}\label{lem:Harncak_ineq}
There is a constant $C > 0$ such that for all $\vep \in (0,1]$, if $(\om_{\vep,\RG} + \ddc v) \w \om_{\vep,\RG}^{n-1} \geq 0$, 
\[
    \sup_\hX v \leq C \lt(1 + \int_\hX |v| \om_{\vep,\RG}^n\rt).
\]
\end{lem}

\begin{proof}
Recall that from Theorem~\ref{thm:unif_Sobolev_Poincare}, we have obtained a uniform Sobolev constant for $(\om_\vep)_{\vep \in (0,1]}$. 
By Theorem~\ref{thm:existence_bdd_Gauduchon}, the metrics $(\om_{\vep,\RG})_{\vep \in (0,1]}$ and $(\om_\vep)_{\vep \in (0,1]}$ are uniformly quasi-isometric; hence there is also a uniform Sobolev constant for the family of metrics $(\om_{\vep,\RG})_{\vep \in (0,1]}$, i.e. for any $q \in (0, \frac{n}{n-1})$, one can find a constant $C_S' > 0$ so that for all $\vep \in (0,1]$, 
\begin{equation}\label{eq:unif_Sobolev_Gaud}
    \forall f \in L^2_1(\hX,\om_{\vep,\RG}),
    \quad 
    \lt(\int_\hX |f|^{2q} \om_{\vep,\RG}^n\rt)^{1/q} 
    \leq C_S' \lt(\int_\hX |\dd f|_{\om_{\vep,\RG}}^2 \om_{\vep,\RG}^n + \int_\hX |f|^2 \om_{\vep,\RG}^n\rt). 
\end{equation}

\smallskip
With \eqref{eq:unif_Sobolev_Gaud}, the conclusion follows from standard Moser's iteration argument. 
For the reader's convenience, we provide some details below. 
By direct computation, we have
\begingroup
\allowdisplaybreaks
{\small
\begin{align*}
    &\int_\hX |\dd v_+^{\frac{p+1}{2}}|_{\om_{\vep,\RG}}^2 \om_{\vep,\RG}^n 
    = n \int_\hX \dd v_+^{\frac{p+1}{2}} \w \dc v_+^{\frac{p+1}{2}} \w \om_{\vep,\RG}^{n-1}
    = \frac{n (p+1)^2}{4} \int_\hX v_+^{p-1} \dd v_+ \w \dc v_+ \w \om_{\vep,\RG}^{n-1} \\
    &= \frac{n (p+1)^2}{4 p} \int_\hX \dd v_+^p \w \dc v_+ \w \om_{\vep,\RG}^{n-1} 
    = \frac{n (p+1)^2}{4 p} \lt(\int_\hX v_+^p (-\ddc v_+) \w \om_{\vep,\RG}^{n-1} - \int_\hX v_+^p \dd v_+ \w \dc \om_{\vep,\RG}^{n-1}\rt) \\
    &\leq \frac{n (p+1)^2}{4 p} \lt(\int_\hX v_+^p \om_{\vep,\RG}^n - \frac{1}{p+1} \int_\hX \dd v_+^{p+1} \w \dc \om_{\vep,\RG}^{n-1}\rt) 
    = \frac{n (p+1)^2}{4 p} \int_\hX v_+^p \om_{\vep,\RG}^n.
\end{align*}
}%
\endgroup
Then
\begingroup
\allowdisplaybreaks
\begin{align*}
    &\lt(\int_\hX v_+^{(p+1)q} \om_{\vep,\RG}^n\rt)^{1/q}
    \leq C_S' \lt(\int_\hX |\dd v_+^{\frac{p+1}{2}}|_{\om_{\vep,\RG}}^2 \om_{\vep,\RG}^n + \int_\hX v_+^{2(p+1)} \om_{\vep,\RG}^n\rt)\\
    &\leq C_S' \lt(\frac{n (p+1)^2}{4 p} \int_\hX v_+^p \om_{\vep,\RG}^n + \int_\hX v_+^{p+1} \om_{\vep,\RG}^n\rt)
    \leq n (p+1) C_S' \max\lt\{\int_{\hX} v_+^{(p+1)} \om_{\vep,\RG}^n, 1\rt\}.
\end{align*}
\endgroup
By standard Moser's iteration process, we obtain 
\[
    \sup_\hX v 
    \leq \sup_\hX v_+ 
    \leq C \lt(1 + \int_\hX v_+ \om_{\vep,\RG}^n\rt)
    \leq C \lt(1 + \int_\hX |v| \om_{\vep,\RG}^n\rt)
\]
where $C$ is independent of $\vep$ and this conclude the proof.
\end{proof}

\subsubsection*{Some Harnack estimates on $\om_{\vep,\RG}$-Hermite--Einstein metrics}
By the result of Li and Yau \cite{Li_Yau_1987} and Theorem~\ref{thm:openness_stability}, for every sufficiently small $\vep > 0$, there is a Hermite--Einstein metric $h_{\vep}$ on $E$ with respect to $\om_{\vep, \RG}$. 
Namely,
\[
    \tr_{\vep,\RG} \ii \Ta(E,h_\vep) = \ld_\vep \Id_E
\]
where $\ld_\vep = \mu_{\om_{\vep,\RG}}(E) \cdot 2 \pi n$.

\smallskip
Let $(L,h_L) = (\det E, \det h_E)$ and set $\Ta_L := \Ta(L,h_L)$. 
Since $\om_{\vep,\RG}$ is Gauduchon, functions in $\ker (\Dt^\ast_{\vep,\RG})$ are constants. 
Standard Hodge theory shows that $\CC^\infty(\hX, \BR) = \ker(\Dt^\ast_{\vep,\RG}) \oplus \im(\Dt_{\vep,\RG})$ in $L^2(\hX,\om_{\vep,\RG})$. 
Thus, the equation
\begin{equation}\label{eq:f_vep}
    \frac{1}{2 \pi n} \lt(\tr_{\vep,\RG} \ii \Ta_L
    + \Dt_{\vep,\RG} f_\vep\rt)
    = \int_\hX \BCC(E) \w \om_{\vep,\RG}^{n-1} 
    = \frac{r \ld_\vep}{2 \pi n}.
\end{equation}
admits a unique smooth solution $f_\vep$ such that $\int_\hX f_\vep \om_{\vep, \RG}^n = 0$.
The Hermite--Einstein metric $h_\vep$ can be expressed as follows 
\[
    h_\vep = h_E e^{-\frac{1}{r} f_\vep} \exp(s_\vep)
\]
where $s_\vep$ is a trace-free endomorphism of $E$. 
We denote by 
\[
    H_\vep := e^{-\frac{1}{r} f_\vep} \exp(s_\vep) 
    =: e^{-\frac{1}{r} f_\vep} S_\vep \in \Herm^+(E,h_E).
\]

\smallskip
Set $D_i$'s irreducible components of the exceptional divisor $D$. 
For each $i$, we take $\sm_i$ a section of $\CO(D_i)$ cutting out the divisor $D_i$. 
Note that there are positive numbers $(a_i)_i$ so that $\pi^\ast \om + \ddc \sum_i \log |\sm_i|_i^{2a_i} \geq \dt \homg$ for some $\dt > 0$ where $|\cdot|_i$ is some hermitian metric on $\CO(D_i)$. 
Set $\vph := \sum_i \log |\sm_i|_i^{2a_i}$. 
Without loss of generality, one may assume that $\vph \leq -1$. 

\smallskip
Using Lemma~\ref{lem:Harncak_ineq}, we follow similar arguments in \cite[Lem.~3.1 and 3.2]{Paun_et_al_2023_HE_sing} to obtain the following estimates: 

\begin{lem}\label{lem:HE_harnack}
There exists a uniform constant $C > 0$ independent of $\vep$ such that 
\[
    - f_\vep \leq C \lt(1 + \int_\hX |f_\vep| \om_{\vep,\RG}^n \rt), 
    \quad
    \frac{C_L}{\dt} \vph + f_\vep \leq C \lt(1 + \int_\hX |f_\vep| \om_{\vep,\RG}^n\rt),
\]
and 
\[
    \log \Tr_{\End} (H_\vep) \leq C \lt(1 + \int_\hX |\log \Tr_{\End}(H_\vep)| \om_{\vep,\RG}^n\rt),
\]
where $C_L > 0$ is a constant such that $\ii \Ta_L \leq C_L \homg$.
\end{lem}

\begin{proof}
By Theorem~\ref{cor:Mihai_et_al_thm1}, $\Ta_L = \Tr_{\End}\Ta(E,h_E) \geq -C \om_{\vep,\RG}$ for some uniform constant $C > 0$ and thus, $\tr_{\vep,\RG} \Ta_L \geq -Cn$. 
Then the equation \eqref{eq:f_vep} yields 
\begin{equation}\label{eq:-f_vep_ineq}
    \Dt_{\om_{\vep,\RG}}(-f_\vep) \geq \tr_{\vep,\RG} \ii \Ta_L - r \ld_\vep 
    \geq - C_1    
\end{equation}
for some uniform constant $C_1>0$.  
Therefore, the first inequality follows directly from \eqref{eq:-f_vep_ineq} and Lemma~\ref{lem:Harncak_ineq}. 

\smallskip
Recall that $\pi^\ast\om + \ddc \vph \geq \dt \om_\hX$ for some $\dt > 0$. 
Let $C_L > 0$ be a constant so that $\ii \Ta_L \leq C_L \homg$; hence, $\ii \Ta_L \leq \frac{C_\Ta}{\dt} (\pi^\ast\om + \ddc \vph)
\leq \frac{C_L}{\dt} (C_G \om_{\vep,\RG} + \ddc \vph)$.
This implies that
\[
    \tr_{\vep,\RG} \ii \Ta_L 
    \leq \frac{C_L}{\dt} (C_G n + \Dt_{\vep,\RG} \vph).
\]
By \eqref{eq:f_vep}, we get 
\begin{equation}\label{eq:log|s_D|_ineq}
    \Dt_{\vep,\RG}\lt(\frac{C_L}{\dt} \vph + f_\vep\rt) \geq n \ld_\vep - \frac{C_L C_G n}{\dt} \geq -C_2
\end{equation}
for some uniform constant $C_2 > 0$.
Note that $\vph$ is quasi-plurisubharmonic, $\int_{\hX} |\vph| \om_{\vep,\RG}^n$ is uniformly bounded for all $\vep \in [0,1]$.
Again, Lemma~\ref{lem:Harncak_ineq}, \eqref{eq:log|s_D|_ineq} and the uniform integrability of $\int_{\hX} |\vph| \om_{\vep,\RG}^n$ yield the second inequality. 

\smallskip
Following the same computation as in \cite[p.~25]{Siu_book}, and the Hermite--Einstein condition, we have
\begingroup
\allowdisplaybreaks
\begin{equation}\label{eq:Siu87_p25_0}
{\small
\begin{split}
    &\Dt_{\vep,\RG} \Tr_{\End}(H_\vep) \\
    &= \Tr_{\End}(\tr_{\vep,\RG} \ii\Ta(E,h_E) \circ H_\vep) 
    - \Tr_{\End}(\tr_{\vep,\RG} \ii\Ta(E,h_\vep) \circ H_\vep)
    + \Tr_{\End}(\tr_{\vep,\RG} \ii [D' H_\vep \circ H_\vep^{-1} \circ (\db_E H_\vep)]) \\
    &= \Tr_{\End}(\tr_{\vep,\RG} \ii\Ta(E,h_E) \circ H_\vep) 
    - \ld_\vep \Tr_{\End}(H_\vep) 
    + \Tr_{\End}(\tr_{\vep,\RG} \ii[D' H_\vep \circ H_\vep^{-1} \circ (\db_E H_\vep)]), 
\end{split}
}%
\end{equation}
\endgroup
and also,
\begin{equation}\label{eq:Siu87_p25_1}
{\small
\begin{split}
    &\Dt_{\vep,\RG} \log \Tr_{\End}(H_\vep) 
    = \frac{\Dt_{\vep,\RG} \Tr_{\End}(H_\vep)}{\Tr_{\End}(H_\vep)} 
    - \frac{|\pl \Tr_{\End}(H_\vep)|_{\om_{\vep,\RG}}^2}{\Tr_{\End}(H_\vep)^2}\\
    &= \frac{\Tr_{\End}(\tr_{\om_\vep,\RG} \ii\Ta(E,h_E) \circ H_\vep)}{\Tr_{\End}(H_\vep)} 
    - \ld_\vep 
    + \frac{\Tr_{\End}(\tr_{\vep,\RG} \ii[D' H_\vep \circ H_\vep^{-1} \circ (\db_E H_\vep)])}{\Tr_{\End}(H_\vep)} 
    - \frac{|\pl \Tr_{\End}(H_\vep)|_{\om_{\vep,\RG}}^2}{\Tr_{\End}(H_\vep)^2}.
\end{split}
}%
\end{equation}
Using Cauchy--Schwarz inequality (cf. \cite[p.~25]{Siu_book} for details), the last two terms satisfy 
\begin{equation}\label{eq:Siu87_p25_2}
    \frac{\Tr_{\End}(\tr_{\vep,\RG} [D' H_\vep \circ H_\vep^{-1} \circ (\db_E H_\vep)])}{\Tr_{\End}(H_\vep)} 
    - \frac{|\pl \Tr_{\End}(H_\vep)|_{\om_{\vep,\RG}}^2}{\Tr_{\End}(H_\vep)^2} \geq 0.
\end{equation}
Then combining \eqref{eq:Siu87_p25_1}, \eqref{eq:Siu87_p25_2} and Theorem~\ref{cor:Mihai_et_al_thm1}, one can obtain 
\begin{equation}\label{eq:Siu87_p25_3}
    \Dt_{\vep,\RG} \log \Tr_{\End} (H_\vep) 
    \geq \frac{\Tr_{\End}(\tr_{\vep,\RG} \ii\Ta(E,h_E) \circ H_\vep)}{\Tr_{\End}(H_\vep)} 
    - \ld_\vep 
    \geq -Cnr - \ld_\vep \geq 
    - C_3
\end{equation}
for a uniform constant $C_3 > 0$, and this allows us to conclude by Lemma~\ref{lem:Harncak_ineq}.
\end{proof}

\subsection{Uniform $L^\infty$-estimates}
In this section, we extend \cite[Thm.~4.1]{Paun_et_al_2023_HE_sing} to the hermitian setting. 

\begin{thm}\label{thm:HE_Linfty_est}
There exists a constant $C > 0$ such that for all $\vep > 0$ sufficiently small,
\[
    \Tr_{\End}(H_\vep) \leq C, 
    \quad -C \leq f_\vep \leq -C \vph
    \quad\text{and}\quad
    \int_\hX \frac{|D' \log H_\vep|^2}{-\vph} \om_\vep^n \leq C.
\]
\end{thm} 

Using \cite[Prop.~3.1]{Nie_Zhang_2018}, we show the following result: 

\begin{prop}\label{prop:HE_gradlog}
Let $\eta_\vep = \log H_\vep$ be the logarithm of the endomorphism defining Hermite--Einstein metric $h_\vep^{\HE}$ (note that $h_\vep^{\HE} = h_E H_\vep$).
Then 
\[
    \int_\hX \Tr_{\End}(\tr_{\vep,\RG} \ii\Ta(E,h_E) \circ \eta_\vep) \om_{\vep,\RG}^n
    + \int_\hX \iprod{\Psi(\eta_\vep) D' \eta_\vep}{D' \eta_\vep}_{h_E,\om_{\vep,\RG}} \om_{\vep,\RG}^n
    = 0,
\]
where $\Psi$ denotes the smooth function 
\[
    \Psi(x,y) 
    = \frac{\exp(x-y) - 1}{x - y}. 
\]
\end{prop}

We recall the notations: 
Let $\psi: \BR \to \BR$ and $\Psi: \BR \times \BR \to \BR$ be two smooth functions. 
Fix $s$ a smooth section of $\End(E)$, which is hermitian with respect to $h_E$. 
On a coordinate chart $U \subset \hX$, one can find unitary frame $(e_\af)_{\af \in \{1, \cdots, r\}}$ diagonalizing $s$, i.e. $s = \sum_\af \ld_\af e_\af \otimes e^\af$. 
Then set 
\[
\psi(s) = \sum_\af \psi(\ld_\af) e_\af \otimes e^\af 
\quad\text{and}\quad
    \Psi(s) A := \sum_{\af,\bt} \Psi(\ld_\af, \ld_\bt) A^\af_\bt e_\af \otimes e^\bt,  
\]
on $U$, for an $\End(E)$-valued $(p,q)$-form $A \underset{\loc.}{=} \sum_{\af,\bt} A^\af_\bt e_\af \otimes e^\bt$. 
Note that the above definitions induced global endomorphisms of $E$.

\begin{proof}[Proof of Proposition~\ref{prop:HE_gradlog}]
Recall that from the Hermite--Einstein equation, we have
\[
    \tr_{\vep,\RG} \ii \Ta(E,h_E) + \tr_{\vep,\RG} \ii [\db_E (H_\vep^{-1} \circ D' H_\vep)] 
    = \ld_\vep \Id_E.
\]
Note that $\eta_\vep = \frac{-f_\vep}{r} \otimes \Id_E + s_\vep$ and $\int_\hX \Tr_{\End}(\eta_\vep) \om_{\vep,\RG}^n = - \int_\hX f_\vep \om_{\vep,\RG}^n = 0$. 
Hence, 
\begin{align*}
    &\int_{\hX} \Tr_{\End}\lt(\tr_{\vep,\RG} \ii \Ta(E,h_E) \circ \eta_\vep\rt) \om_{\vep,\RG}^n 
    + \int_\hX \Tr_{\End}\lt(\tr_{\vep,\RG} \ii[\db_E(H_\vep^{-1} \circ D' H_\vep)] \circ \eta_\vep\rt) \om_{\vep,\RG}^n \\
    &= \int_{\hX} \Tr_{\End}\lt(\ld_\vep \Id_E \circ \eta_\vep\rt) \om_{\vep,\RG}^n 
    = \ld_\vep \int_\hX \Tr_{\End}(\eta_\vep) \om_{\vep,\RG}^n
    = 0.
\end{align*}
By Stokes' theorem, 
\begingroup
\allowdisplaybreaks
\begin{align*}
    &\int_\hX \Tr_{\End}\lt(\tr_{\vep,\RG} (\db_E (H_\vep^{-1} \circ D' H_\vep)) \circ \eta_\vep\rt) \om_{\vep,\RG}^n \\
    &= n \int_\hX \db \Tr_{\End} \lt((H_\vep^{-1} \circ D' H_\vep) \circ \eta_\vep \w \om_{\vep,\RG}^{n-1}\rt) 
    + n \int_\hX \Tr_{\End} \lt((H_\vep^{-1} \circ D' H_\vep) \w \db_E \lt(\eta_\vep \om_{\vep,\RG}^{n-1}\rt)\rt) \\
    &= n \int_\hX \Tr_{\End} \lt((H_\vep^{-1} \circ D' H_\vep) \w \db_E \eta_\vep \w \om_{\vep,\RG}^{n-1}\rt) 
    + n \int_\hX \Tr_{\End} \lt((H_\vep^{-1} \circ D' H_\vep) \circ \eta_\vep \w \db \om_{\vep,\RG}^{n-1} \rt).
\end{align*}
\endgroup
From the standard computation (see e.g. \cite[(3.8)]{Nie_Zhang_2018}), one has $\Tr_{\End}(H_\vep^{-1} \circ D' H_\vep \circ \eta_\vep) = \Tr_{\End}(\eta_\vep \circ D' \eta_\vep)$.
Since $\om_{\vep,\RG}$ is Gauduchon, 
{\small
\[
    \int_\hX \Tr_{\End}\lt((H_\vep^{-1} \circ D' H_\vep) \circ \eta_\vep \w \db \om_{\vep,\RG}^{n-1}\rt) 
    = \frac{1}{2} \int_\hX \pl \Tr_{\End}(\eta_\vep^2) \w \db \om_{\vep,\RG}^{n-1}
    = -\frac{1}{2} \int_\hX \Tr_{\End}(\eta_\vep^2) \ddb \om_{\vep,\RG}^{n-1} = 0,
\]
}%
and thus, 
\[
    \int_\hX \Tr_{\End} \lt(\tr_{\vep,\RG} [\db(H_\vep^{-1} \circ D' H_\vep)] \circ \eta_\vep\rt) \om_{\vep,\RG}^n 
    = \int_\hX \Tr_{\End} \lt((H_\vep^{-1} \circ D' H_\vep) \w \db_E \eta_\vep \w \om_{\vep, \RG}^{n-1}\rt).
\]

\smallskip
Let $(e_\af)_\af$ be a local unitary frame and $(A^\af_\bt)_{\af,\bt}$ be the connection form of $D'$ with respect to $(e_\af)_\af$. 
Write $\eta_\vep = \sum_\af \ld_\af e_\af \otimes e^\af$. 
Then we have
\[
    H_\vep^{-1} \circ D' H_\vep 
    = \sum_\af \pl \ld_\af e_\af \otimes e^\af 
    + \sum_{\af,\bt} (\exp(\ld_\bt - \ld_\af) - 1) A_\bt^\af e_\af \otimes e^\bt
\]
and 
\[
    \db_E \eta_\vep 
    = \sum_\af \db \ld_\af e_\af \otimes e^\af 
    + \sum_{\af,\bt} (\ld_\bt - \ld_\af) \bar{A}_\bt^\af e_\af \otimes e^\bt.
\]
Then 
\begin{align*}
    &\tr_{\vep,\RG} \Tr_{\End}(H_\vep^{-1} \circ D' H_\vep \w \db \eta_\vep)
    = \sum_\af |\pl \ld_\af|_{\om_{\vep,\RG}}^2  
    + \sum_{\af,\bt} (\exp(\ld_\af - \ld_\bt) - 1) (\ld_\af - \ld_\bt) |A_\bt^\af|_{\om_{\vep,\RG}}^2\\
    &= \sum_\af |\pl \ld_\af|_{\om_{\vep,\RG}}^2  
    + \sum_{\af,\bt} \frac{(\exp(\ld_\af - \ld_\bt) - 1)}{\ld_\af - \ld_\bt} (\ld_\af - \ld_\bt)^2 |A_\bt^\af|_{\om_{\vep,\RG}}^2 
    = \iprod{\Psi(\eta_\vep) D' \eta_\vep}{D' \eta_\vep}_{h_E, \om_{\vep,\RG}}.
\end{align*}
This completes the proof.
\end{proof}

\begin{proof}[Proof of Theorem~\ref{thm:HE_Linfty_est}]
Since $\om_{\vep,\RG}$ is Gauduchon, by Stokes' theorem, for any smooth function $\chi$ and any $\vep \in (0,1]$, we have the following
\begin{equation}\label{eq:HE_infty_IPP}
    \int_\hX |\dd \chi|_{\om_{\vep,\RG}}^2 \om_{\vep,\RG}^n
    = - n \int_\hX \chi \ddc \chi \w \om_{\vep,\RG}^{n-1}
    = - \int_\hX \chi (\Dt_{\vep,\RG} \chi) \om_{\vep,\RG}^n.
\end{equation}
Recall that from \eqref{eq:f_vep}, we have
\[
    \Dt_{\vep,\RG} f_\vep 
    = r \ld_\vep - \tr_{\vep,\RG} \ii \Ta_L 
\]
and $\int_\hX f_\vep \om_{\vep,\RG}^n = 0$. 
By the uniform Poincar\'e inequality in Theorem~\ref{thm:unif_Sobolev_Poincare} and \eqref{eq:HE_infty_IPP},
\begin{align*}
    \int_\hX |f_\vep|^2 \om_{\vep,\RG}^n 
    &\leq C_P \int_\hX |\dd f_\vep|^2_{\om_{\vep,\RG}} \om_{\vep,\RG}^n 
    = -C_P \int_\hX f_\vep (\Dt_{\vep,\RG} f_\vep) \om_{\vep,\RG}^n \\
    &= n C_P \int_\hX f_\vep \ii \Ta_L \w \om_{\vep,\RG}^{n-1} - r \ld_\vep C_P \int_\hX f_\vep \om_{\vep,\RG}^n\\
    &\leq n C_P \int_\hX \max\{\sup(-f_\vep), \sup(f_\vep)\} C_L \homg \w \om_{\vep,\RG}^{n-1}
    + r |\ld_\vep| C_P \int_\hX |f_\vep| \om_\vep^n,
\end{align*}
where $C_L > 0$ is a constant such that $\ii \Ta_L \leq C_L \homg$. 
Then Lemma~\ref{lem:HE_harnack} implies that
\begin{equation}\label{eq:f_vep_int_bdd}
\begin{split}
    \int_\hX |f_\vep|^2 \om_{\vep,\RG}^n 
    &\leq n C_P \int_\hX \lt(2C + 2C \int_\hX |f_\vep| \om_{\vep,\RG}^n - \frac{C_L}{\dt} \vph\rt) C_L \homg \w \om_{\vep,\RG}^{n-1}
    + r |\ld_\vep| C_P \int_\hX |f_\vep| \om_{\vep,\RG}^n\\
    &\leq C' \lt(1 + \int_\hX |f_\vep| \om_{\vep,\RG}^n\rt),  
\end{split}
\end{equation}
for a constant $C' > 0$ independent of $\vep$. 
In the second inequality, we use the quasi-plurisubharmonicity of $\vph$ and Theorem~\ref{thm:existence_bdd_Gauduchon} to have a uniform integral bound on $\int_\hX |\vph| \homg \w \om_{\vep,\RG}^{n-1}$. 
By H\"older inequality, \eqref{eq:f_vep_int_bdd} shows that
\[
    \int_\hX |f_\vep|^2 \om_{\vep,\RG}^n \leq C' + C' \lt(\int_\hX |f_\vep|^2 \om_{\vep,\RG}^n\rt)^{1/2}
\]
and hence $(\int_\hX |f_\vep|^2 \om_\vep^n)$ is uniformly bounded. 
Then Lemma~\ref{lem:HE_harnack} again implies that
\begin{equation}\label{eq:HE_infty_f_bdd}
    -C_1 \leq f_\vep \leq -C_1 \vph,
\end{equation}
for some uniform constant $C_1 > 0$ and this shows the desired estimates on $f_\vep$ in Theorem~\ref{thm:HE_Linfty_est}.

\smallskip
We now proceed to control $\Tr_{\End} H_\vep$. 
To achieve this, we make the following claim, which we shall verify in Appendix~\ref{sec:pf_claim} using Simpson's contradiction argument. 

\begin{claim}\label{claim:int_bdd_logTrS}
There exist a constant $C > 0$ independent of $\vep > 0$ such that
\[
    \int_\hX \log \Tr_{\End} S_\vep \om_{\vep,\RG}^n \leq C_2.
\]
\end{claim}

Since $\log \Tr_{\End} H_\vep = \frac{-f_\vep}{r}+ \log \Tr_{\End} S_\vep$, from Lemma~\ref{lem:HE_harnack}, 
\begin{align*}
    \log \Tr_{\End} H_\vep 
    &\leq C \lt(1 + \int_\hX |\log \Tr_{\End} H_\vep| \om_{\vep,\RG}^n\rt)
    \leq C \lt(1 + \int_\hX \frac{|f_\vep|}{r} + \log \Tr_{\End} S_\vep \om_{\vep,\RG}^n\rt)\\
    &\leq C \lt(1 + \frac{C_1}{r} \int_\hX -\vep \om_{\vep,\RG} + C_2\rt)
    \leq C_3,
\end{align*}
for some $C_3 > 0$ independent of $\vep$. 
This provides a uniform upper bound of $\Tr_{\End} H_\vep$. 

\smallskip
By Proposition~\ref{prop:HE_gradlog}, we have 
\begin{align*}
    \int_\hX \iprod{\Psi(\eta_\vep) D' \eta_\vep}{D' \eta_\vep}_{h_E, \om_{\vep,\RG}} \om_{\vep,\RG}^n
    &= - \int_\hX \Tr_{\End} (\tr_{\vep,\RG} \ii\Ta(E,h_E) \circ \eta_\vep) \om_{\vep,\RG}^n \\
    &\leq rC_1 \int_\hX (-\vph) |\tr_{\vep,\RG} \Ta(E,h_E)|_{h_E} \om_{\vep,\RG}^n.
\end{align*}
From Theorem~\ref{thm:existence_bdd_Gauduchon}, $|\tr_{\vep,\RG} \Ta(E,h_E)| \om_{\vep,\RG}^n \leq n |\Ta(E,h_E) \w \om_{\vep,\RG}^{n-1}| \leq C_E \homg^n$ for some $C_E > 0$ independent of $\vep$.
Since $\eta_\vep = \frac{-f_\vep}{r} \otimes \Id_E + \log S_\vep$ and $\Tr_{\End} \log S_\vep = 0$, one can deduce $\log S_\vep \leq (C + \frac{f_\vep}{r}) \otimes \Id_E \leq C(1-\frac{\vph}{r}) \otimes \Id_E$ and thus $\log S_\vep \geq -C(1-\vph) \otimes \Id_E$. 
This implies that 
\[
    C\vph \otimes \Id_E \leq \eta_\vep \leq C \Id_E
\]
and thus, $\ld_{\max} - \ld_{\min} \leq C(1-\vph) \leq -2C\vph$. 
Then 
\[
    \iprod{\Psi(\eta_\vep) D' \eta_\vep}{D' \eta_\vep}
    \geq \frac{1 - e^{2C \vph}}{-2C \vph} |D' \eta_\vep|^2
    \geq \frac{1-e^{-2C}}{-2C\vph} |D' \eta_\vep|^2.
\]
Hence, we obtain the following uniform estimate
\[
    \int_\hX \frac{|D' \eta_\vep|^2_{h_E, \om_{\vep,\RG}}}{-\vph} \om_{\vep,\RG}^n \leq C.
\]
This completes the proof.
\end{proof}

With Theorem~\ref{thm:HE_Linfty_est}, we are now able to conclude Theorem~\ref{bigthm:HE_metric}. 

\begin{proof}[Proof of Theorem~\ref{bigthm:HE_metric}]
By the local regularity theory (cf. \cite[Appx.~C-E]{Jacob_Walpuski_2018} for details) for Hermite--Einstein equation, after extracting a subsequence, $H_\vep$ converges locally smoothly on $\hX \setminus D$ to an $H$ hermitian endomorphism solving the Hermite--Einstein equation there. 
Remark that as the proof presented in \cite[Appx.~C-E]{Jacob_Walpuski_2018} is purely local, the non-K\"ahler condition does not affect the situation here.  
The estimates in Theorem~\ref{bigthm:HE_metric} also follow directly. 

\smallskip
It remains to verify the uniqueness. 
Suppose that $h_1$ and $h_2$ are two singular Hermite--Einstein metrics both satisfying the normalization and the estimates in Theorem~\ref{bigthm:HE_metric}. 
Write $h_i = h_0 H_i$ and $H_i = e^{-f_i/r} S_i$ for $i = 1, 2$. 
From the estimates obtained in Theorem~\ref{thm:HE_Linfty_est}, there is a constant $C > 0$ such that 
\begin{equation}\label{eq:uniqueness_Linfty_est}
    -C \leq f_i \leq -C \vph, 
    \quad\text{and}\quad
    C\vph \otimes \Id_\CE \leq \log S_i \leq -C\vph \otimes \Id_\CE
    \quad\text{on } X^\circ.
\end{equation}
The lower bound of $S_i$ comes from its control from above and $\Tr_{\End} \log S_i = 0$.

\smallskip
Consider $\mathfrak{H} := h_1^\ast \otimes h_2$ a hermitian metric on $\End(\CE_{|X^\circ})$. 
Then by a direct computation, we have
\begin{equation}\label{eq:uniqueness_Lap_ID}
    \Dt_{\om_\RG} |\Id_\CE|_{\mathfrak{H}}^2 
    = 2 |D'_{\mathfrak{H}} \Id_\CE|_{\mathfrak{H}}^2,
    \quad\text{and}\quad
    \Dt_{\om_\RG} \log |\Id_\CE|_{\mathfrak{H}}^2
    = \frac{2 |D'_{\mathfrak{H}} \Id_\CE|_{\mathfrak{H}}^2}{|\Id_\CE|_{\mathfrak{H}}^2}
    - \frac{|\pl |\Id_\CE|_{\mathfrak{H}}^2|_{\om_\RG}^2}{|\Id_\CE|_{\mathfrak{H}}^4}
    \geq 0
\end{equation}
on $X^\circ$. 
Set $A := \log_+ |\Id_\CE|_{\mathfrak{H}}^2$.
Then we also have $\Dt_{\om_\RG} A \geq 0$.
Since $Z = X \setminus X^\circ$ has complex codimension at least $2$ in $X$, there exists a sequence of cutoff functions $(\chi_k)_k$ supported on $X^\circ$ such that $\chi_k$ increases to the characteristic function of $X^\circ$. 
Moreover, these functions satisfy $\int_X |\dd \chi_k|_{\om}^4 + |\ddc \chi_k|_\om^2 \om^n < C$ for some constant $C > 0$ independent of $k$ and $\int_X |\dd \chi_k|_{\om}^3 + |\ddc \chi_k|_\om^{3/2} \om^n \to 0$ as $k \to +\infty$. 
One can obtain these cutoffs by the standard argument in \cite[Lem.~1]{Schoen_Simon_1981} or \cite[Sec.~4.7]{EG_1992}.  
For the reader's convenience, we also provide the construction of these cutoff functions in Lemma~\ref{lem:cutoffs}.

\smallskip
Now, we want to find a uniform bound for $(\int_X |\dd (\chi_k A)|_{\om_\RG}^2 \om_\RG^n)_k$. 
By Stokes' theorem and the Gauduchon condition of $\om_\RG$, we have
\begin{equation}\label{eq:uniqueness_IPP}
{\footnotesize
\begin{split}
    \int_X |\dd (\chi_k A)|_{\om_\RG}^2 \om_\RG^n 
    &= n \int_X A^2 \dd \chi_k \w \dc \chi_k \w \om_\RG^{n-1} + n \int_X \chi_k^2 \dd A \w \dc A \w \om_\RG^{n-1}
    + 2n \underbrace{\int_X \chi_k A \dd\chi_k \w \dc A \w \om_\RG^{n-1}}_{= \frac{1}{2} \int_X A \dd \chi_k^2 \w \dc A \w \om_\RG^{n-1}}\\
    &= n \int_X A^2 \dd \chi_k \w \dc \chi_k \w \om_\RG^{n-1} + n \int_X \chi_k^2 \dd A \w \dc A \w \om_\RG^{n-1} \\
    &\qquad- n \int_X \chi_k^2 \dd A \w \dc A \w \om_\RG^{n-1} 
    - n \underbrace{\int_X \chi_k^2 A \ddc A \w \om_\RG^{n-1}}_{\geq 0} 
    - n \underbrace{\int_X \chi_k^2 A \dd A \w \dc \om_\RG^{n-1}}_{= \frac{1}{2} \int_X \chi_k^2 \dd A^2 \w \dc \om_\RG^{n-1}}\\
    &\leq \int_X A^2 |\dd \chi_k|_{\om_\RG}^2 \om_\RG^n 
    + n \int_X A^2 \chi_k \dd \chi_k \w \dc \om_\RG^{n-1}.
\end{split}
}%
\end{equation}
Recall that $\om_\RG^{n-1} = \rho \om^{n-1}$ where $\rho \in \CC^\infty(X^\reg) \cap L^\infty(X^\reg)$ is a bounded Gauduchon factor. 
Since $\rho_0$ can be obtained as a limit of $(\rho_\vep)_{\vep \in (0,1]}$ as $\vep \to 0$, we have
\[
    \int_X |\dd\rho|_\om^2 \om^n 
    = \int_\hX |\dd\rho|_{\pi^\ast\om}^2 \pi^\ast \om^n
    \leq \liminf_{\vep \to 0} \int_{\hX} |\dd \rho_\vep|_{\om_\vep}^2 \om_\vep^n
    \leq \liminf_{\vep \to 0} \frac{nB}{2} \int_\hX \rho_\vep^2 \om_\vep^n
\]
where $B>0$ is a uniform constant such that $-B \om_\vep \leq \ddc \om_\vep^{n-1} \leq B \om_\vep^n$ (see also Step 1 in the proof of Theorem~\ref{thm:existence_bdd_Gauduchon}). 
The last inequality comes from \cite[Lem.~2.4]{Pan_2022}.
Hence, $\rho \in L^2_1(X,\om)$.
By Cauchy--Schwarz inequality, 
\begin{equation}\label{eq:uniqueness_CS}
\begin{split}
    n \dd \chi_k \w \dc \om_\RG^{n-1} 
    &= n \dd \chi_k \w \dc \rho \w \om^{n-1}
    + n (n-1) \rho \dd \chi_k \w \dc \om \w \om^{n-2} \\ 
    &\leq |\dd \chi_k|_\om |\dd \rho|_\om \om^n + c_n C_G |\dd \chi_k|_\om |\dd \om|_\om \om^n
\end{split}
\end{equation}
for some dimensional constants $c_n$. 
Remark that the torsion term $|\dd\om|_\om$ is bounded. 
Indeed, since $\om$ extends smoothly to a metric $\om'$ under a local embedding $j: U \hookrightarrow \BC^N$ where $U$ is an open set on $X$, we have 
\[
    |\dd\om|_\om^2(x)
    = |\dd \wom_{|j(U^\reg)}|_{\wom_{|j(U^\reg)}}^2 (j(x))
    \leq |\dd\wom|_{\wom}^2 (j(x))
\] 
for any $x \in U^\reg$. 
The torsion $|\dd\wom|_{\wom}^2$ is under control near a neighborhood of $j(U)$. 
Set $C_\om > 0$ an upper bound of $|\dd\om|_\om$.
Then combining \eqref{eq:uniqueness_IPP} and \eqref{eq:uniqueness_CS}, by H\"older inequality, one can obtain 
{\footnotesize
\begin{align*}
    \int_X |\dd (\chi_k A)|_{\om_\RG}^2 \om_\RG^n 
    &\leq \lt(\int_X A^6 \om_\RG^n\rt)^{1/3} \lt(\int_X |\dd \chi_k|_{\om_\RG}^3 \om_\RG^n\rt)^{2/3} + \lt(\int_X A^{12} \om^n\rt)^{1/6} \lt(\int_X |\dd \chi_k|_{\om}^3 \om^n\rt)^{1/3} \lt(\int_X |\dd\rho|^2_\om \om^n\rt)^{1/2} \\
    &\quad + c_n C_G C_\om \lt(\int_X |\dd \chi_k|_\om^2 \om^n\rt)^{1/2} \lt(\int_X A^4 \om^n\rt)^{1/2}. 
\end{align*}
}%
Note that \eqref{eq:uniqueness_Linfty_est} shows that $A \in L^p(X)$ for all $p > 1$. 
Since $\om_\RG$ is quasi-isometric to $\om$, and $(\int_X |\dd \chi_k|_\om^3 \om^n)_k$ converges to zero, the RHS in the above inequality goes to zero. 
This shows that $A$ is a constant and thus $|\Id_\CE|_{\mathfrak{H}}^2 \in L^\infty(X^\circ)$.

\smallskip
Write $h_2(s, t) = h_1(\Phi(s), t)$ where $\Phi = H_1^{-1} \circ H_2$. 
Using cutoff functions $(\chi_k)_k$, \eqref{eq:uniqueness_Lap_ID}, boundedness of $|\Id_\CE|_{\mathfrak{H}}^2$ and Stokes theorem, one can infer $D'_{\mathfrak{H}} \Id_\CE = 0$ on $X^\circ$.  
This implies that $\Phi$ is a holomorphic section of $\End(\CE_{|X^\circ})$.  
Since $\CE$ is reflexive, $\End(\CE)$ is also a reflexive sheaf. 
As $Z = X \setminus X^\circ$ has codimension at least $2$, $\Phi$ extends to a global section of $\End(\CE)$, which we still denote by $\Phi$. 
By standard argument (see e.g. \cite[Ch.~V, Cor.~7.14]{Kobayashi_book}), $\Phi$ is a constant multiple of the identity map. 
\end{proof}

\begin{lem}\label{lem:cutoffs}
Let $X$ be an $n$-dimensional compact complex space, and let $S \subset X$ be a complex analytic subset which has codimension at least $k$.
Fix $\om$ a hermitian metric on $X$. 
Then there exists cutoff functions $(\chi_i)_{i \in \BN^\ast}$ increasing to $\1_{X \setminus S}$ and a constant $C>0$ such that $\int_X |\dd \chi_i|_\om^{2k} + |\ddc \chi_i|_\om^{k} \om^n \leq C$ for all $i \in \BN^\ast$ and $\lim_{i \to +\infty} \int_X |\dd \chi_i|_\om^{2k-\vep} + |\ddc \chi_i|_\om^{k-\vep} \om^n = 0$ for any $\vep > 0$ small. 
\end{lem}

\begin{proof}
We follow the construction in \cite[pages 751-752]{Schoen_Simon_1981} and \cite[pages 89-90]{Wickramasekera_2008}.
It suffices to verify the existence of cutoffs locally. 
Let $U$ be an open neighborhood of $x \in X$, and let $j: (U,x) \hookrightarrow (\BC^N,0)$ be a local embedding. 
Denote by $\BB^N(z,r)$ ball of radius $r$ in $\BC^N$ centered at $z$ and $\om$ the Euclidean metric on $\BC^N$.
Since $S' := j(S \cap U) \cap \BB^N(0,2)$ is compact, for each $i \in \BN^\ast$, there exist a finite number $m^{(i)}$ and balls $\BB^N(z_\af^{(i)}, r_\af^{(i)})$ for $\af \in \{1, 2, \cdots, m^{(i)}\}$ with $z_\af^{(i)} \in S'$ such that $S' \subset \bigcup_{\af=1}^{m^{(i)}} \BB^N(z_\af^{(i)}, r_\af^{(i)}) =: \CS^{(i)}$, 
\begin{equation}\label{eq:Hausdorff_sing}
    \sum_{\af=1}^{m^{(i)}} v_{2(n-k)} (2r_\af^{(i)})^{2(n-k)}
    \leq K 
    := 1 + \CH^{2(n-k)}(S')
\end{equation}
and $r_\af^{(i)} \leq \tau^{(i)}$, where $v_d$ is the volume of the unit ball in $\BR^d$ and $\CH^d$ is the $d$-dimensional Hausdorff measure.
Here $\tau^{(1)} := \tau \in (0, \frac{1}{8})$ is a sufficiently small fixed number so that \eqref{eq:Hausdorff_sing} holds, and $\tau^{(i)} = \frac{1}{4} \dist(S', \BC^N \setminus \CS^{(i-1)})$. 

\smallskip
For each $i \in \BN^\ast$ and each $\af \in \{1, \cdots, m^{(i)}\}$, let $\psi_\af^{(i)}$ be a $\CC^1$-function defined near $j(U)$ such that 
{\small
\begin{align*}
    \psi_\af^{(i)}
    =
    \begin{cases}
        0 & \text{on } j(U) \cap \BB^N(z_j^{(i)}, r_j^{(i)}) \\
        1 & \text{on } j(U) \setminus \BB^N(z_j^{(i)}, 2r_j^{(i)})
    \end{cases}, 
    \quad
    0 \leq \psi_\af^{(i)} \leq 1, \quad 
    |\dd \psi_\af^{(i)}|_{\om} \leq \frac{2}{r_\af^{(i)}}, \quad\text{and}\quad
    |\ddc \psi_\af^{(i)}|_{\om} \leq \frac{2}{(r_\af^{(i)})^2} \text{ on $j(U)$}.
\end{align*}
}%
Consider $\chi_i = \prod_{\af=1}^{m^{(i)}} \psi_\af^{(i)}$ and set $V^{(i)} = \bigcup_{\af=1}^{m^{(i)}} \BB^N(z_\af^{(i)}, 2r_\af^{(i)})$. 
By the construction, we have $\chi_i \equiv 0$ on $\CS^{(i)}$, $\chi_i \equiv 1$ on $j(U) \setminus V^{(i)}$, the support of $|\dd \chi_i|_\om$ and $|\ddc \chi_i|_\om$ on $j(U)$ is contained in $j(U) \cap (V^{(i)} \setminus V^{(i+1)})$. 
From \eqref{eq:Hausdorff_sing}, one can conclude that
{\small
\[
    \int_{j(U)} |\dd \chi_i|_\om^{2k} \om^n 
    \leq \int_{j(U) \cap V^{(i)}} |\dd \chi_i|_\om^{2k} \om^n
    \leq \sum_{\af = 1}^{m^{(i)}} \lt(\frac{2}{r_\af^{(i)}}\rt)^{2k} \CH^{2n}(j(U) \cap \BB^N(z_\af^{(i)}, 2r_\af^{(i)})) 
    \leq \sum_{\af=1}^{m^{(i)}} 4^{k+n} (r_\af^{(i)})^{2(n-k)}
    \leq c K
\] 
}%
where $c = c(n,k)$ is a constant.
Similarly, 
\[
    \int_{j(U)} |\ddc \chi_i|_\om^{k} \om^n
    \leq c' K
\] for a constant $c' > 0$. 
By H\"older inequality and $\Vol_\om(V^{(i)}) \to 0$ as $i \to +\infty$, one can deduce that $\lim_{i \to +\infty} \int_X |\dd \chi_i|_\om^{2k-\vep} + |\ddc \chi_i|_\om^{k-\vep} \om^n = 0$ for all $\vep > 0$ small. 
\end{proof}

\appendix

\section{Proof of Claim~\ref{claim:int_bdd_logTrS}}\label{sec:pf_claim}

As in \cite[Sec.~4.2.2]{Paun_et_al_2023_HE_sing}, we proceed with a proof by contradiction using Simpson's approach \cite[p.~885-889]{Simpson_1988}. 
Suppose that there exists $\vep_k \to 0$ and $\dt_k \to 0$ such that
\begin{equation}\label{eq:claim_hypo}
    \int_\hX \dt_k \log \Tr_{\End} S_{\vep_k} \om_{\vep_k, \RG}^n = 1.
\end{equation}
For simplicity, we shall denote by $S_k := S_{\vep_k}$, $H_k := H_{\vep_k}$, and $\om_{k,\RG} := \om_{\vep_k,\RG}$.

\smallskip
Recall that from \eqref{eq:HE_infty_f_bdd}, we have 
\[
    -C_1 \leq f_k \leq -C_1 \vph.
\]
By Lemma~\ref{lem:HE_harnack}, 
\begin{equation}\label{eq:HE_Linfty_case2_f_H_bd}
    \log \Tr_{\End} (H_k) \leq C \lt(1 + \int_\hX \lt(\frac{|f_k|}{r} + \log \Tr_{\End} S_k\rt) \om_{k,\RG}^n\rt)
    \leq \frac{C'}{\dt_k} 
\end{equation}
for some uniform constant $C' > 0$, and thus 
\[
    \dt_k \log \Tr_{\End} S_k \leq C' + \frac{\dt_k f_k}{r} 
    \leq C' + \frac{- \dt_k C_1 \vph}{r}.
\]
This implies that all eigenvalues of $\log S_k$ are less than $\frac{C'}{\dt_k} + \frac{-C_1 \vph}{r}$. 
Since $\log S_k$ has zero trace, the eigenvalues are bounded below by $-\frac{rC'}{\dt_k} + C_1\vph$.
Set 
\[
    u_k := \dt_k \eta_k = -\dt_k \frac{f_k}{r} \otimes \Id_E + \dt_k \log S_k.
\]
Then we have 
\begin{equation}\label{eq:HE_Linfty_case2_u_bd}
    |u_k|_{h_E} \leq \frac{-\dt_k C_1 \vph}{r} + rC' - \dt_k C_1 \vph.
\end{equation}

\smallskip
Following \cite[Lem.~4.6]{Paun_et_al_2023_HE_sing}, one can derive: 

\begin{lem}\label{lem:simpson_1}
There exists a subsequence of $(u_k)_k$ converging weakly in $L^2_1$ towards a limit $u_\infty$ on any compact subset $K \Subset \hX \setminus D$ such that the following hold: 
\begin{enumerate}
    \item The endomorphism $u_\infty$ is non-trivial and it belongs to $L^2_{1}(\hX, \pi^\ast \om_\RG)$;
    \item Let $\Phi: \BR \times \BR \to \BR_{>0}$ be a smooth function such that $\Phi(a,b) \leq \frac{1}{a-b}$ for any $a > b$. 
    Then 
    \[
        0 \geq \int_\hX \iprod{\Phi(u_\infty) D' u_\infty}{D' u_\infty}_{h_E,\pi^\ast\om_\RG} \pi^\ast \om_\RG^n 
        + \int_\hX \Tr_{\End}(\tr_{\pi^\ast \om_\RG} \ii\Ta(E,h_E) \circ u_\infty) \pi^\ast \om_\RG^n.
    \]
\end{enumerate}
\end{lem}

\begin{proof}
By Proposition~\ref{prop:HE_gradlog}, we have
\[
    \frac{1}{\dt_k} \int_\hX \iprod{\Psi(u_k/\dt_k) D'u_k}{D'u_k}_{h_E,\om_{k,\RG}} \om_{k,\RG}^n 
    + \int_\hX \Tr_{\End}(\tr_{k,\RG} \ii\Ta(E,h_E) \circ u_k) \om_{k,\RG}^n 
    = 0.
\]
Note that
\[
    \frac{1}{\dt_k} \Psi(u_k/\dt_k) D'u_k  
    = \sum_{\af,\bt} \frac{\exp\lt(\frac{\ld_\af-\ld_\bt}{\dt_k}\rt) - 1}{\ld_\af-\ld_\bt} A^\af_\bt e_\af \otimes e^\bt
\]
Fix $K$, a compact subset of $\hX \setminus D$. 
For sufficiently large $k$, we have 
\[
    \frac{1}{\dt_k} \int_\hX \iprod{\Psi(u_k/\dt_k) D'u_k}{D'u_k}_{h_E,\om_{k,\RG}} \om_{k,\RG}^n
    \geq \int_K \iprod{\Phi(u_k) D'u_k}{D'u_k}_{h_E,\om_{k,\RG}} \om_{k,\RG}^n
\]
and thus, 
\[
    \int_K \iprod{\Phi(u_k) D'u_k}{D'u_k}_{h_E,\om_{k,\RG}} \om_{k,\RG}^n 
    + \int_\hX \Tr_{\End}(\tr_{k,\RG} \Ta(E,h_E) \circ u_k) \om_{k,\RG}^n 
    \leq 0
\]
for all $k$ large enough.
Note that on $K$, ${u_k}_{|K}$ has uniformly bounded eigenvalues. 
Therefore, we obtain a uniform gradient estimate as follows: 
\[
    \int_K |D'u_k|_{h_E,\om_{k,\RG}}^2 \om_{k,\RG}^n \leq C
\]
for all $k \gg 1$.
Since $(\om_{k,\RG})_k$ are uniformly quasi-isometric on $K$, we can extract a subsequence of $(u_k)_k$ converging weakly in $L^2_1(K)$ and strongly in $L^2(K)$ towards $u_\infty \in L^2_1(K, \pi^\ast\om_\RG)$. 
Enlarging $K$ and using a diagonal argument, we obtain a subsequence such that $u_k$ converges weakly in $L^2_{1,\loc}(\hX \setminus D, \pi^\ast \om_\RG)$ and strongly in $L^2_\loc(\hX \setminus D, \pi^\ast \om_\RG)$.

\smallskip
We are now going to show that $u_\infty$ is not identically zero.
Note that
\begin{equation}\label{eq:HE_Linfty_case2_S_ub}
\begin{split}
    \dt_k \log \Tr_{\End} S_k
    &\leq \dt_k \log r + \dt_k |s_k|_{h_E} 
    \leq \dt_k \log r + |u_k|_{h_E} + \dt_k |f_k| 
\end{split}
\end{equation}
Thus, from \eqref{eq:claim_hypo} and \eqref{eq:HE_Linfty_case2_S_ub}, 
\[
    \int_{\hX} |u_k|_{h_E} \om_{k,\RG}^n 
    \geq 1 - \dt_k \log r - 2 \dt_k \int_\hX |f_k| \om_{k,\RG}^n
    \geq 1 - \dt_k \log r + 2 C_1 \dt_k \int_\hX \vph \om_{k,\RG}^n
    \xrightarrow[k \to +\infty]{}1.
\]
Since $\dt_k f_\vep$ converging to $0$ in $L^p(\hX, \homg)$ for any $p>0$ and $\CC^\infty_{\loc}(\hX \setminus D)$, we have
\[
    - \frac{\dt_k f_k}{r} \Id_E + \dt_k s_k
    = u_k \xrightarrow[k \to +\infty]{L^2_{\loc}(\hX \setminus D)} u_\infty. 
\]
From \eqref{eq:HE_Linfty_case2_u_bd}, we know that for all $\epsilon > 0$, there is an open neighborhood $U^\epsilon$ of $D$ such that
\begin{equation}
    \int_{U^\epsilon} |u_k|_{h_E} \om_{k,\RG}^n \leq \epsilon
\end{equation} 
Therefore, if $u_\infty \equiv 0$, taking $K = \hX \setminus U^{1/2}$, we get 
\[
    0 = \int_K |u_\infty|_{h_E} \pi^\ast \om_{\RG}^n  
    = \lim_{k \to +\infty} \int_K |u_k|_{h_E} \om_{k,\RG}^n  
    = \int_\hX |u_k|_{h_E} \om_{k,\RG}^n - \int_{U^{1/2}} |u_k|_{h_E} \om_{k,\RG}^n
    \geq 1 - 1/2 > 0 
\]
which shows a contradiction to $u_\infty \equiv 0$.

\smallskip
We now do the estimate for $\Phi(u_\infty)$. 
The term $\int_\hX \Tr_{\End}(\tr_{k,\RG} \ii \Ta(E,h_E) u_k) \om_{k,\RG}^n$ is continuous in $u_k \in L^2_1$. 
Therefore, for any $\epsilon > 0$, 
\[
    \int_\hX \iprod{\Phi(u_k) D'u_k}{D'u_k}_{h_E, \om_{k,\RG}} \om_{k,\RG}^n
    + \int_\hX \Tr_{\End}(\tr_{\pi^\ast \om_\RG} \ii \Ta(E,h_E) u_\infty) \pi^\ast \om_{\RG}^n 
    \leq \epsilon
\]
for some $k \gg 1$.
Since $u_k \to u_\infty$ in $L^2_{0,b}(K)$ for any $K \Subset \hX \setminus D$, \cite[Prop~4.1~(b)]{Simpson_1988} shows that $\Phi^{1/2}(u_k) \to \Phi^{1/2}(u_\infty)$ in $\Hom(L^2, L^q)(K)$ for any $q < 2$. 
As $D' u_k$ are uniformly bounded in $L^2$, after extracting a subsequence, one can find compact subsets $K_k \subset \hX \setminus D$ such that $K_k \subset K_{k+1}$, $\cup_k K_k = \hX \setminus D$ and 
\[
    \lt(\int_{K_k} |\Phi(u_k)^{1/2} D'u_k|_{h_E, \om_{k,\RG}}^q \om_{k,\RG}^n\rt)^{2/q} 
    + \int_\hX \Tr_{\End}(\tr_{\pi^\ast \om_\RG} \ii \Ta(E,h_E) u_\infty) \pi^\ast \om_{\RG}^n 
    \leq 2 \epsilon.
\]
We also have $\Phi(u_k)^{1/2} D'u_k \rightharpoonup \Phi(u_\infty)^{1/2} D'u_\infty$ weakly in $L^q(K)$ for any $K \subset \hX \setminus D$ compact.
By the lower semi-continuity of the norm with respect to weak convergence and enlarging $K$, 
\[
    \lt(\int_\hX |\Phi(u_\infty)^{1/2} D'u_\infty|_{h_E, \pi^\ast \om_\RG}^q \pi^\ast \om_{\RG}^n\rt)^{2/q} 
    + \int_\hX \Tr_{\End}(\tr_{\pi^\ast \om_\RG} \ii \Ta(E,h_E) u_\infty) \pi^\ast \om_{\RG}^n 
    \leq 2 \epsilon.
\]
The above estimate holds for all $q < 2$; hence one obtains the desired estimate in Lemma~\ref{lem:simpson_1} by taking $q \to 2$.
\end{proof}

Then applying Uhlenbeck and Yau's technique, one can construct a destabilizing subsheaf. 
We extract the following two lemmas 
from \cite[p.~886-889]{Simpson_1988} where the K\"ahler condition does not affect the argument.

\begin{lem}[{\cite[Lem.~5.5,~5.6]{Simpson_1988}}]\label{lem:simpson_2}
The following properties hold: 
\begin{itemize}
    \item 
    The eigenvalues of $u_\infty$ are constant, meaning there exist constants $\ld_1, \cdots \ld_r$ which are eigenvalues of $u_\infty(x)$ for almost every $x \in X$. 
    These eigenvalues $(\ld_\af)_\af$ are not all equal.
    \item 
    If $\Phi: \BR \times \BR \to \BR$ satisfies $\Phi(\ld_\af, \ld_\bt) = 0$ wherever $\ld_\af > \ld_\bt$, then $\Phi(u_\infty) D' u_\infty = 0$.
\end{itemize}
\end{lem}

Without loss of generality, one may assume that $\ld_1 \leq \ld_2 \leq \cdots \leq \ld_r$. 
Let $\gm_1 < \gm_2 < \cdots < \gm_l$ be the distinct eigenvalues of $u_\infty$, where $2 \leq l \leq r$.
For each $1 \leq i \leq l$, choose a smooth function $p_{\gm_i}$ satisfying
\[
    p_{\gm_i}(x) = \begin{cases}
        1, & x \leq \gm_i\\
        0, &x \geq \gm_{i+1}.
    \end{cases}
\]
Set $\pi_i := p_{\gm_i}(u_\infty)$. 
Then $\pi_i$ sends $e_\af$ to $e_\af$ with $\ld_\af \leq \gm_i$ and $e_\bt$ to zero if $\ld_\bt > \gm_i$. 

\begin{lem}[{\cite[Lem.~5.7]{Simpson_1988}}]\label{lem:simpson_3}
For each $i \in \{1,\cdots,l-1\}$, $\pi_i$ induces a proper subsheaf $\CF_i$ of $E$, and 
\[
    \frac{\deg_{\pi^\ast \om_\RG} (\CF_i)}{\rk(\CF_i)} 
    \geq \frac{\deg_{\pi^\ast \om_\RG}(E)}{\rk(E)}
\]
for some $i \in \{1, \cdots, l-1\}$.
\end{lem}

Lemma~\ref{lem:simpson_3} yields a contradiction to the stability of $E$, and this completes the proof of Claim~\ref{claim:int_bdd_logTrS}.

\bibliographystyle{smfalpha_new}
\bibliography{biblio}
\end{document}